\documentclass[12pt]{amsart}

\usepackage{amscd}
\usepackage{marvosym}
\usepackage{amsmath}
\usepackage{amsgen}
\usepackage{amsbsy}
\usepackage{amsopn}
\usepackage{amsthm}
\usepackage{amscd}
\usepackage{amsfonts}
\usepackage{amssymb}
\usepackage{mathrsfs}
\usepackage[backref=page]{hyperref}
\usepackage[capitalize]{cleveref}

\hypersetup{
  colorlinks   = true,
  urlcolor     = teal,
  linkcolor    = teal,
  citecolor   = teal
}
\usepackage{mathtools}
\usepackage{ifsym}
\usepackage[usenames,dvipsnames]{xcolor}

\usepackage[color=Aquamarine,textsize=footnotesize,obeyFinal]{todonotes}
\setuptodonotes{inline}

\usepackage[marginratio=1:1,height=8.5in,width=5.5in,tmargin=1.25in]{geometry}
\usepackage[normalem]{ulem}
\usepackage{enumerate,enumitem}
\usepackage{tikz-cd}

\usepackage{graphicx}
\usepackage{comment}

\usepackage{latexsym}
\usepackage{MnSymbol}

\newcommand{\mc}{\mathcal}

\newcommand{\mf}{\mathfrak}

\newtheorem{Theorem}{Theorem}[section]
\newtheorem{Lemma}[Theorem]{Lemma}
\newtheorem{lemma}[Theorem]{Lemma}
\newtheorem{Corollary}[Theorem]{Corollary}

\newtheorem{Proposition}[Theorem]{Proposition}

\newtheorem{Definition}[Theorem]{Definition}

\newtheorem{Example}[Theorem]{Example}
\newtheorem{Remark}[Theorem]{Remark}

\newcommand{\Vol}{\operatorname{Vol}}
\DeclareMathOperator{\vol}{vol}

\newcommand{\pa}{\partial}
\newcommand{\del}{\partial}
\newcommand{\op}{\operatorname}
\newcommand{\abs}[1]{\left| #1 \right|}
\newcommand{\norm}[1]{\left\| #1 \right\|}
\newcommand{\inner}[1]{\left\langle #1 \right\rangle}
\newcommand{\floor}[1]{\left\lfloor #1 \right\rfloor}

\newcommand{\set}[1]{\left\{ #1 \right\} }
\newcommand{\tensor}{\otimes}
\newcommand{\grad}{\nabla}
\newcommand{\til}{\widetilde}
\newcommand{\of}{\circ}

\newcommand{\h}{{hyp}}
\newcommand{\mres}{%
  \,\raisebox{-.127ex}{\reflectbox{\rotatebox[origin=br]{-90}{$\lnot$}}}\,}%

\DeclareMathOperator{\Jac}{Jac}
\DeclareMathOperator{\Hess}{Hess}

\DeclareMathOperator{\RCD}{\ensuremath{\mathsf{RCD}}}
\DeclareMathOperator{\CAT}{\ensuremath{\mathsf{CAT}}}
\DeclareMathOperator{\CBB}{\ensuremath{\mathsf{CBB}}}
\DeclareMathOperator{\diam}{diam}
\DeclareMathOperator{\bary}{bar}

\DeclareMathOperator{\argmin}{argmin}
\DeclareMathOperator{\ind}{ind}
\DeclareMathOperator{\pre}{pre}
\DeclareMathOperator{\esssup}{ess-sup}
\DeclareMathOperator{\directsum}{\bigoplus}

\newcommand{\mX}{\mf{m}}
\newcommand{\tmX}{\bar{\mf{m}}}

\newcommand{\R}{{\bf R}}
\newcommand{\Z}{{\bf Z}}
\newcommand{\N}{{\bf N}}
\newcommand{\Q}{{\bf Q}}
\newcommand{\C}{{\bf C}}
\newcommand{\HH}{{\bf H}}
\newcommand{\HK}{{\bf H}_{\bf{K}}}
\newcommand{\OO}{{\bf O}}

\newcommand{\eps}{\epsilon}
\newcommand{\epsv}{\varepsilon}
\newcommand{\Ga}{\Gamma}
\newcommand{\ga}{\gamma}

\newcommand{\cout}[1]{}

\renewcommand{\hat}{\widehat}
\renewcommand{\bar}{\overline}

\def\be#1\ee{\begin{align}\begin{split} #1 \end{split}\end{align}}
\def\beq#1\eeq{\begin{align*}\begin{split} #1 \end{split}\end{align*}}

\title{Volume entropy and rigidity for $\RCD$-spaces}

\author[Connell]{Chris Connell}
\address[Chris Connell]{Department of Mathematics, Indiana University, Bloomington, IN 47405 }
\email{connell@iu.edu}
\author[Dai]{Xianzhe Dai}
\address[Xianzhe Dai]{Department of Mathematics, UCSB, Santa Barbara CA 93106}
\email{dai@math.ucsb.edu}
\author[N\'u\~nez-Zimbr\'on]{Jes\'us N\'u\~nez-Zimbr\'on}
\address[Jes\'us N\'u\~nez-Zimbr\'on]{Facultad de Ciencias, Universidad Nacional Aut\'{o}noma de M\'{e}xico, Coyoac\'{a}n, Mexico}
\email{nunez-zimbron@ciencias.unam.mx}
\author[Perales]{Raquel Perales}
\address[Raquel Perales]{Centro de Investigaci\'on en Matem\'aticas,
De Jalisco s/n, Valenciana, Guanajuato, Gto. Mexico. 36023}
\email{raquel.perales@cimat.mx}
\author[Su\'arez-Serrato]{Pablo Su\'arez-Serrato}
\address[Pablo Su\'arez-Serrato]{ Geometric Intelligence Laboratory, Electrical and Computer Engineering, University of California, Santa Barbara, USA,
{\it and}
Instituto de Matem\'aticas, Universidad Nacional Aut\'onoma de M\'exico UNAM, M\'exico Tenochtitlan}
\email{pablo@im.unam.mx}
\author[Wei]{Guofang Wei}
\address[Guofang Wei]{Mathematics Department, University of California, Santa Barbara, CA 93106}
\email{wei@math.ucsb.edu}

\numberwithin{equation}{section}

\begin{document}

\begin{abstract}
We develop the barycenter technique of Besson--Courtois--Gallot so that it can be applied on $\RCD$ metric measure spaces.
Given a continuous map $f$ from a non-collapsed $\RCD(-(N-1),N)$ space $X$ without boundary to a locally symmetric $N$-manifold we show a version of BCG's entropy-volume inequality. The lower bound involves homological and homotopical indices which we introduce. We prove that when equality holds and these indices coincide $X$ is a locally symmetric manifold, and $f$ is homotopic to a Riemannian covering whose degree equals the indices.
Moreover, we show a measured Gromov--Hausdorff stability
of $X$ and $Y$ involving the homotopical invariant.
As a byproduct, we extend a Lipschitz volume rigidity result of Li--Wang to $\RCD(K,N)$ spaces without boundary.
Finally, we include an application of these methods to the study of Einstein metrics on $4$-orbifolds.
\end{abstract}

\maketitle

\tableofcontents

\section{Introduction}\label{sec:Intro}

We define the  {\em volume entropy} $h(Z)$ of a metric measure space $(Z,d,\mX)$ (cf. Reviron \cite{Reviron} and Besson--Courtois--Gallot--Sambusetti\cite{BCG-Samb}) as
\[
h(Z)=\limsup_{R\to\infty}\frac{\log \mX(B(x,R))}{R},
\]
where $B(x,R)$ is the geodesic ball of radius $R$ in $Z$ centered at $x\in Z$. By the triangle inequality, the value of $h(Z)$ is independent of the point $x$.

Traditionally, the universal cover $\til{Z}$ is used in the definition of volume entropy.
On the one hand, our convention differs in that we compute the volume on the space $Z$ itself and not on $\til{Z}$.
In other words, the usual volume entropy in our notation would be indicated by $h(\til{Z})$.
On the other hand, we will crucially use intermediate covers similarly to Sambusetti \cite{Sambusetti-1999}.

Whenever $Z$ is a length space with a cocompact group of isometries, the $\limsup$ in the definition of $h(Z)$ may be replaced by a limit (see Manning \cite{M} for the manifold case, and \cite[Proposition 3.3]{BCG-Samb} for metric measure spaces).

In this paper we extend the seminal minimal volume entropy rigidity results of Besson--Courtois--Gallot  to $\RCD$-spaces.
These are metric measure spaces$(X,d,\mf{m})$ with a {\it synthetic} lower Ricci bound and a dimension upper bound.
In particular, these results are obtained by extending the {\it barycenter technique}, as developed by Besson--Courtois--Gallot \cite{Besson-Courtois-Gallot:95,BCG-rend}---and in their collaboration with B\`essieres \cite{Bess-BCG}---for manifolds with Ricci curvature bounded below.
See the following subsection for some historical comments about this technique.
Sturm \cite{Sturm03} established various results for barycenters in CAT$(0)$ spaces which will be useful in this extension.
We also employ in a critical way some of the machinery developed by Sambusetti \cite{Sambusetti-1999} to generalize \cite{Besson-Courtois-Gallot:95} and \cite{BCG-acta}.

$\RCD$-spaces were shown to admit universal covers by Mondino and Wei \cite{Mondino-Wei}. Recently, Wang \cite{Wang2022} showed that their universal covers are also semi-locally simply connected.
When the measure of $X$ equals the $N$-dimensional Hausdorff measure, $\mf{m} = \mc{H}^{N}$, then
$(X,d,\mf{m})$ is called {\it non-collapsed}, and  the {\em boundary} $\pa X$ of $X$ can be defined.
See Section \ref{sec:Prelims} for more details about $\RCD$-spaces.
In a complementary direction, we have showed a maximal volume entropy rigidity for $\RCD$-spaces \cite{stormmod}.

Certain important aspects of the theory of manifolds are lacking for $\RCD$ spaces. One of these is the degree theory of maps, which we need to state our results.
This motivates the definitions of {the homotopy invariants}
$\ind_\pi(f)$ and $\ind_H(f)$ defined as follows.
Given a continuous map $f:X\to Y$ between topological spaces, with topological $\dim Y = N$,
we define the {\em fundamental index of $f$} by:
\begin{equation}
\ind_{\pi}(f)=\begin{cases}[\pi_1(Y):f_*\pi_1(X)] & \text{if } [\pi_1(Y):f_*\pi_1(X)]<\infty \\
    0 & \text{otherwise.}\end{cases}  \label{def-f-ind}
\end{equation}
Likewise, we define the {\em  homological index of $f$} by:
\begin{equation}
\ind_H(f)=\begin{cases}[H_N(Y):f_*H_N(X)] & \text{if }[H_N(Y):f_*H_N(X)]<\infty \\
    0 & \text{otherwise.}\end{cases}  \label{def-h-ind}
\end{equation}
Here, $H_N(X)$ denotes the $N$-th singular homology of $X$ with $\Z$-coefficients.

When $Y$ admits a universal cover, we always have that $\ind_\pi(f)$ divides $\ind_H(f)$ (see Proposition \ref{prop:pre-pi1} and Remark \ref{remark:fundamental and homological indeces}
where we observe that $\ind_H(f)$ may be strictly larger than $\ind_\pi(f)$). In the case that $X$ and $Y$ are closed oriented manifolds we have $\ind_H(f)=\abs{\deg(f)}$. Thus, our definition of $\ind_H(f) $  can be seen as a generalization of $\abs{\deg(f)}$ to the non-smooth setting.

In what follows, for any continuous map $f:X\to Y$, we let $\bar{X}$ be the cover of $X$ corresponding to the subgroup $\ker f_*<\pi_1(X)$ where $f_*:\pi_1(X)\to\pi_1(Y)$ is the induced map on fundamental groups. In particular, $\pi_1(\bar{X})=\ker f_*$ and $\Ga:=\pi_1(X)/\ker f_*$ acts on $\bar{X}$ by deck transformations. Observe that $\bar{X}$ is the smallest cover for which there is a lift of $f$ to a map $\til{f}:\bar{X}\to\til{Y}$, where $\til{Y}$ is the universal cover of $Y$.

Our first main result is:
\begin{Theorem}\label{thm:entrig}
Let $K \in \R$, $N \in \N$ with $N\geq 3$, $(X,d,\mc{H}^{N})$ be an $\RCD(K,N)$ space without boundary, and $Y$ be a closed orientable negatively curved locally symmetric space of dimension $N$.
Then,  for any continuous map $f:X\to Y$,
\begin{equation}\label{eq:Thm1-ineq}
h(\bar{X})^{N}\ \mc{H}^{N} (X)\geq \ind_H(f)\ h(\til{Y})^{N}\ \mc{H}^{N}(Y).
\end{equation}
Moreover, if we have equality and $\ind_H(f)=\ind_\pi(f)$, then $X$ is isometric to a locally symmetric manifold and $f$ is homotopic to a Riemannian cover of degree $\ind_{\pi}(f)$, after possibly dilating the metric on $X$.
\end{Theorem}

\begin{Remark}
Note that since the universal cover $\til{X}$ is a cover of $\bar{X}$,
we have $h(\til{X})\geq h(\bar{X})$, and thus the above theorem implies the usual entropy estimate with the inequality replaced by
\[h(\til{X})^{N}\ \mc{H}^{N} (X)\geq \ind_H(f)\ h(\til{Y})^{N}\ \mc{H}^{N}(Y),\]
and similarly for the equality case (compare with the manifold case in Section 2 of \cite{Sambusetti-1999}).
\end{Remark}

Our second main theorem relaxes the condition on the target manifold $Y$, and
removes the volume entropy by normalization to obtain a volume rigidity theorem.

\begin{Theorem}\label{thm:schwarz}
Let $K \in  \R$, $N \in \N$ with $N\geq 3$.
Let $(X,d_X,\mc{H}^{N})$ be an $\RCD(-(N-1),N)$ space and $(Y,d_Y,\mc{H}^{N})$ be a compact orientable space that is both locally $\CAT(-1)$ and an $\RCD(K,N)$ space with $\pa Y=\emptyset$. Then, for any continuous map $f:X\to Y$,
\begin{equation}\label{eq:Thm2-ineq}
\mc{H}^{N}(X)\geq \ind_H(f)\ \mc{H}^{N}(Y).
\end{equation}
Moreover, if $X$ has no boundary
and
\begin{equation}\label{eq:Thm2-equal}
\mc{H}^{N}(X)=\ind_\pi(f) \mc{H}^{N}(Y),
\end{equation}
then $X$ and $Y$ are isometric to hyperbolic manifolds and $f$ is homotopic to a degree $\ind_\pi(f)$ Riemannian cover with respect to the constant curvature $-1$ metrics.
\end{Theorem}

For examples of spaces that satisfy the hypotheses of Theorem \ref{thm:schwarz}, see Remarks
\ref{rmrk-Exam} and \ref{rmrk-moreExam}.

\begin{Remark}
Theorems \ref{thm:entrig}  and \ref{thm:schwarz} generalize
the results \cite[Th\'eor\`eme Principal]{Besson-Courtois-Gallot:95},  \cite[Th\'eor\`eme p.734]{Besson-Courtois-Gallot:95}, and \cite[Corollaire 1.4]{BCG-acta} of Besson--Courtois--Gallot in the case of maps for which $\ind_H(f)=\ind_\pi(f)$, e.g. when $\abs{\deg(f)}=1$.

In particular, \cref{thm:entrig} implies Mostow Rigidity in the rank one case by applying it to $X=Y$ and $f$ as the identity.
\end{Remark}

Consider the special case of Theorem \ref{thm:schwarz} when $f$ is a homotopy equivalence and the target is a hyperbolic manifold of constant curvature $-1$. In this case, we obtain:

\begin{Corollary}\label{cor:isomhyp}
Let $(X,d,\mc{H}^{N})$ be an  $\RCD(-(N-1),N)$ space without boundary and $M_\h$ a closed  hyperbolic
$N$-manifold of constant curvature $-1$.  If $X$ and $M_\h$ are homotopy equivalent, then
$$\mc{H}^{N}(X) \ge \mc{H}^{N}(M_\h).$$
Moreover, \cout{if the measure $\mf{m}$ on $X$ is in the $N$-dimensional Hausdorff
class, then} equality occurs if and only if $X$ is isometric to $M_\h$.
\end{Corollary}

\begin{Remark}
This generalizes the main result of Storm \cite{Storm02} (and \cite[Theorem 8.5]{Storm07}).
\end{Remark}

One of our principal applications is the following result which can be seen as an extension of Theorem 1.3 of Bessi\`eres--Besson--Courtois--Gallot \cite{Bess-BCG} to the $\RCD$ setting, see Remark \ref{rmrk-differencehyp} for details about the explicit differences between our result and theirs, and Remark \ref{rmrk-Examples}.

\bigskip

\begin{Theorem}\label{thm:stab}
Given any integer $N\geq 3$ and constants $K\in \R$ and $D>0$,
there is an $\epsv_0=\epsv_0(N,K,D)>0$ such that the following holds. Suppose $(X,d_X,\mc{H}^{N})$ is an $\RCD(-(N-1),N)$ space with $\pa X=\emptyset$ and $\diam(X)<D$,
and $(Y,d_Y,\mc{H}^{N})$ is a compact locally $\CAT(-1)$ non-collapsed $\RCD(K,N)$ space with  $\pa Y=\emptyset$.
If $f:X\to Y$ is any continuous map with $f_*H_N(X,\Z)\neq 0$, then for any positive $\epsv<\epsv_0$ we have,
\begin{equation*}
\mc{H}^{N}(X) \leq \ind_\pi(f)\ (1+\epsv)\mc{H}^{N}(Y)
\end{equation*}
if and only if $X$ and $Y$ are homeomorphic to hyperbolic manifolds with metrics $\delta(\epsv)$ measured Gromov--Hausdorff close to the hyperbolic ones.
Moreover,
$f$ is homotopy equivalent to a covering map of degree $\ind_\pi(f)$.
\end{Theorem}

The very final step of the proof of Theorems \ref{thm:entrig} and \ref{thm:schwarz} relies on the following result, which is an extension of the Lipschitz volume rigidity theorem of Li--Wang \cite{LiWang-limits}.
\begin{Theorem}[Lipschitz volume Rigidity]\label{thm:LV-rigid}
	Assume  $K\in \R$ and $N\geq 3$ is an integer.
	Let $(X,d_X,\mc{H}^{N})$ and $(Y,d_Y,\mc{H}^{N})$ be $\RCD(K,N)$ spaces without boundary. Suppose there is a 1-Lipschitz map
	$f:X \rightarrow Y$ with
	$$\mc{H}^{N}(X)=\mc{H}^{N}(f(X)),$$
	then $f$ is an isometry with respect to the intrinsic metrics of $X$ and $f(X).$ In particular, if $f$ is also onto, then $X$ is isometric to $Y$.
\end{Theorem}

Another application of our methods recovers and gives a potential extension of a result by Besson--Courtois--Gallot \cite[Th\'eor\`eme 9.6]{Besson-Courtois-Gallot:95} about uniqueness of Einstein metrics on hyperbolic $4$-manifolds to certain Einstein $4$-orbifolds, (see \cref{cor:Einstein} below).

\subsection{The barycenter technique and organization of the paper}

A method for extending conformal homeomorphisms of the circle to the unit disc was introduced by Douady--Earle \cite{DE}.
Their ideas are at the root of the barycenter technique further developed by Besson--Courtois--Gallot \cite{Besson-Courtois-Gallot:95, BCG-rend, BCG-acta}, used to solve a conjecture by Gromov about compact locally symmetric spaces. This family of ideas consists of a way to smooth maps within a homotopy class with certain nice properties akin to harmonic maps.
Improvements of these techniques to work on finite volume manifolds were achieved by Boland--Connell--Souto \cite{BCS}.
Further work by Storm removed the bounded geometry hypothesis \cite{St1}, he also expanded the possible spaces where this approach can be used to include Alexandrov spaces \cite{Storm02}, and to certain other singular spaces \cite{Storm07}.
A recent variation of this theme has also been successfully applied to manifolds modelled on products of copies of the hyperbolic plane by Merlin \cite{Mer}.
Analogous results were obtained for manifolds with Ricci curvature bounded below, as well as related stability resultsm by Bessi\`eres--Besson--Courtois--Gallot \cite{Bess-BCG}.
Other formulations of closely related maps arising from the barycenter construction, and their uses, were described by the first named author and others e.g. \cite{Con,ConSS19,Storm07,BCG-Samb,Lafont-Schmidt}.
The maps arising from the barycenter construction are often referred to as  ``natural maps.''
Recently, Song studied a version of the Plateau problem for group homology adapting Besson--Courtois--Gallot's ideas to work on metric currents in an infinite-dimensional Hilbert--Riemannian manifold \cite{Song}.

Our contributions here increase the scope of applicability of the barycenter method to the more general setting of $\RCD$-spaces. As with most developments of the barycenter method, our work relies on bounding the Jacobian of the resulting natural maps using the entropy and the dimension.
We achieve this in several steps as our metric setting is quite different than the Riemannian one.
The proof of the rigidity statements in \cref{thm:entrig} and \cref{thm:schwarz} after achieving a 1-Lipschitz map are also necessarily completely different.
The proof we present below is different from the proof of Besson--Courtois--Gallot's celebrated main result \cite[TH\'EOR\`EME PRINCIPAL, pg.734]{Besson-Courtois-Gallot:95}, which relies on the spherical volume.
In their more recent work, Bessi\`eres--Besson--Courtois--Gallot  \cite{Bess-BCG} used the barycenter method on limits of sequences of manifolds with Ricci curvature bounded below.
From a bird's eye view, our approach to proving \cref{thm:stab} is somewhat similar to the general strategy of proof of \cite[Theorem 1.3]{Bess-BCG}.
Nevertheless, our proof differs in several key points from their original arguments.
Moreover, the proof of \cref{thm:schwarz} relies on some of the same technology used to prove \cref{thm:entrig} for $\RCD$-spaces.
All of which requires the following innovations:
\begin{enumerate}
     \item  Important aspects of manifold theory are lacking for $\RCD$ spaces. One of them is Brower's degreee theory. We replace the standard notion of degree of a smooth map with our fundamental ($\ind_{\pi}$) and homological ($\ind_{H}$) indices, developed in detail in \cref{sec:degree}.
     \item A crucial step in the barycenter method is the ability to effectively bound the norm of the Jacobian of a map. We define the Jacobian using the $\RCD$ structure of $X$ in \cref{subsec:Jacobian}, and prove the key estimate we need in \cref{prop:Jac_est}.
     \item In \cref{sec:Proofs}, we exploit the Wasserstein distance to prove the natural maps $F_s$ are Lipschitz in \cref{lem:Lips}.

     \item  For the rigidity results (in the cases of equality), we rely on an extended version of the result of Li--Wang (see \cref{thm:LV-rigid}).
\end{enumerate}

We will now explain the organization of the paper.
In Section \ref{sec:Prelims} we  describe the tools from the $\RCD$-spaces theory that we will need.
In Section \ref{sec:degree} we develop a homotopy invariant of maps from these metric spaces to manifolds which plays the role of a weak notion of absolute degree.
In Section \ref{sec:Proofs} we extend the barycenter machinery to our context and establish the necessary estimates needed to prove our results using these tools.
Note that while we do not need to use the second-order theory of $\RCD$ spaces coming from heat kernel estimates as we did in our previous work on maximal entropy rigidity \cite{stormmod}, we do need to deal with the inherent lack of smoothness of these spaces which must be controlled under Lipschitz assumptions alone.
 In Section \ref{subsec:Fs-Lip} we establish the Lipschitz continuity of the natural map by utilizing Sturm's results (see \cref{lem:barLip}).

 Then,
in Section \ref{sec:inequalities_proof},  we prove the inequality statements in Theorems \ref{thm:entrig} and \ref{thm:schwarz}.
Section \ref{sec:vol-rigid} contains the proof of our augmented version of the volume rigidity theorem of Li--Wang.
This is used in Section \ref{sec:equality_proofS}
to establish the equality (rigidity) statements in Theorems \ref{thm:entrig} and \ref{thm:schwarz}.

The stability result of \cref{thm:stab} is proved in \cref{sec:stability_proof}.
Finally, sections \ref{sec:proof4.8} and \ref{sec:FinW1infty} contain the proofs of two key   results (\cref{prop:Jac_est} and \cref{prop:F_converge}) that we need for the proof of Theorems \ref{thm:entrig} and \ref{thm:schwarz}.

{\bf{Acknowledgments:}} The authors thank The University of California Institute for Mexico and the United States (UC MEXUS) for support for this project through the grant CN-16-43. CC was partially supported by a grant from the Simons Foundation \#210442. XD was partially supported by Simons Foundation.
GW was partially supported by NSF DMS grant 2104704 and 2403557. JNZ is supported by PAPIIT-UNAM project  IN101322 and IA103925.
 GW would like to thank Xingyu Zhu for helpful conversation and Shouhei Honda for bringing \cite[Proposition 3.6]{Honda-Sire} to our attention.

 A final draft of this manuscript is based upon work supported by the National Science Foundation under Grant No. DMS-1928930, while XD, RP and GW were in residence at the Mathematical Sciences Research Institute in Berkeley, California, during the Fall semester of 2024.

PSS and CC warmly thank the Max Planck Institute for Mathematics in Bonn for providing an excellent working environment and support.
PPS thanks the Geometric Intelligence Laboratory in UC Santa Barbara, for the great atmosphere and support during the final writing stages, and the Casa Matem\'atica Oaxaca for hosting a wonderful workshop in 2022.

\section{Preliminaries}\label{sec:Prelims}

In this section we review several concepts required for the arguments in the proofs of our results.
We begin by recalling the basic notions of weak upper gradient and Sobolev functions on metric measure spaces.
We assume the reader to be familiar with the basic notions of $\RCD$ spaces
and we only present the relevant elements of the theory that we require in the rest of the article, such as the stratification into regular and singular sets, the corresponding chart decomposition due to Mondino--Naber, the relevant results on non-collapsed spaces, and the definition of the boundary.
We then proceed to recall the definition of the Jacobian in this general context, using the coarea formula for metric measure spaces due to Ambrosio--Kirchheim \cite{Ambrosio-Kirchheim}.
(We use Reichel's formulation \cite{Reichel09} and see also \cite[Theorem 1.4]{Ka}.)

\subsection{$\RCD$ spaces and their boundary}\label{sec:weakgrad}
Let $(X,d,\mf{m})$ be a complete separable metric space with a Radon measure $\mf{m}$. We say that a curve $\ga\in C([0,1],X)$ is \textit{absolutely continuous} if there exists a map $f\in L^1([0,1])$ satisfying
\[
d(\gamma_t,\gamma_s)\leq \int_s^t f(r)\, dr
\]
for every $t,s\in [0,1]$ with $s<t$. The \textit{metric speed} of an absolutely continuous curve $\gamma$ is the limit
\[
\abs{\dot\ga_t}:=\limsup_{h\to 0} \frac{d(\ga_{t+h},\ga_t)}{h},
\]
which exists for a.e. $t$. Moreover, the map $\abs{\dot\ga_t}$ is integrable and it is the minimal map that can be chosen as $f$ in the definition of absolutely continuous curve.

Let $\mc{P}(X)$ denote the space of Borel probability measures on $X$, and let $e_t:C([0,1],X)\to X$ be the {\it evaluation map at time} $t$ on curves given by $e_t(\ga)=\ga_t$.
A {\em test plan} is a measure $\pi\in \mc{P}(C([0,1],X))$ such that
$$(e_t)_*\pi\leq C(\pi) \mf{m}$$ for all $t\in[0,1]$ and some constant $C(\pi)>0$, and
\[
\int\mkern-13mu\int_0^1 \abs{\dot{\gamma_t}}^2 \,dt\, d\pi(\gamma) < \infty.
\]

Recall that a {\it weak upper gradient} for a function $f$
is a non-negative function  $G\in L^2(X,\mf{m})$ such that for all test plans $\pi\in \mc{P}(C([0,1],X))$, we have
\[
\int \abs{f(\gamma_1)-f(\gamma_0)}d\pi(\gamma)\leq \int\mkern-13mu\int_0^1 G(\gamma_t)\abs{\dot{\gamma_t}} \, dt\, d\pi(\gamma).
\]
The set of weak upper gradients of $f$ is a convex and closed subset of $L^2(X,\mf{m})$ (see \cite[Proposition 2.1.11]{GigliPasqualettoBook}).
As $L^2(X,\mf{m})$ is a Hilbert space, it follows that there exists a unique pointwise minimal weak upper gradient of $f$ which is denoted by $\abs{\grad f}$.
The {\it Sobolev $(1,2)$-space of $X$}, $W^{1,2}(X,d,\mf{m})$ is the space of elements of $L^2(X,\mf{m})$ for which
$\abs{\grad f}$ exists and such that $\norm{f}_{1,2}:= \norm{f}_2+\norm{\abs{\grad f}}_2$ is bounded.

Let us now recall that $(X,d,\mf{m})$ is an {\it $\RCD(K,N)$ space} for given $K\in\R$ and $N\in[1,\infty]$ if it is an {\it infinitesimally Hilbertian space}, that is $W^{1,2}(X,d,\mf{m})$ is a Hilbert space, and $X$ satisfies the {\it curvature-dimension condition $\mathsf{CD}(K,N)$} (see for example   \cite{GigliPasqualettoBook, Lott-Villani09, Sturm06I, Sturm-MM2} for an account of the basic theory).
For the rest of the section, we assume that $(X,d,\mf{m})$ is an $\RCD(K,N)$-space.

Let $x\in {\rm supp}(\mf{m})$, and $r\in (0,1)$. Consider the rescaled and normalized pointed metric measure space $(X,r^{-1}d,\mf{m}_{r}^{x}, x)$, with:
\[
\mf{m}_{r}^{x}:=\left( \int\limits_{B(x,r)} 1- \frac{1}{r} d( \cdot , x)\, \mf{m} \right)^{-1}\, \mf{m}
\]

\begin{Definition}
Let $(X,d,\mf{m})$ be a metric measure space and $x\in {\rm supp}(\mf{m})$. A pointed metric measure space $(Y, d_{Y},\mf{m}_{Y}, y)$ is called a (metric measure) tangent space to $(X,d,\mf{m})$ at $x$ if there exists a sequence of radii $r_{i} \searrow 0$, so that
\[
(X,r_i^{-1}d,\mf{m}_{r_i}^{x}, x)\to (Y, d_{Y},\mf{m}_{Y}, y),
\]
as $i\to \infty$ in the pointed measured Gromov-Hausdorff topology.
\end{Definition}

The collection of all metric measure tangent spaces at a point $x\in X$ is denoted by ${\rm Tan}(X,d,\mf{m},x)$.
The {\it $k$-dimensional regular set} $\mc{R}^k$ is the set of points $x\in X$ such that ${\rm Tan}(X,d,\mf{m},x)$ consists of a single space, isomorphic (that is, isometric where the isometry is measure-preserving) to the $k$-dimensional Euclidean space $(\R^{k}, d_{Euc},\omega_{k}^{-1}\mathcal{L}^k,0)$.
Here, $d_{Euc}$ is the Euclidean distance, $\mathcal{L}^k$ is the $k$-dimensional Lebesgue measure, and $\omega_{k}$ is the volume of the unit ball in $\R^k$.
It follows from the Bishop--Gromov volume comparison theorem that $\mc{R}^k=\emptyset$ for $k>N$.

A structural result for $\RCD$-spaces  obtained by Mondino--Naber shows that $X$ is stratified by the $k$-th strata $\mc{R}^{k}$ \cite{MN}. Contributions by De Philipis--Gigli \cite{DePhilippisGigli}, Gigli--Mondino--Rajala \cite{GigliMondinoRajala}, Gigli--Pasqualetto \cite{GigliPasqualetto} and Kapovitch--Mondino \cite{Kapovitch-Mondino}, strengthened this decomposition showing that each $\mc{R}^k$ is $k$-rectifiable, and that the measure is mutually absolutely continuous to the $k$-dimensional Hausdorff measure $\mc{H}^k$ (see also Theorem 1.18 of Ambrosio--Honda--Tewodrose \cite{AmbrosioHondaTewodrose}).
Moreover, it has been shown by Bru\`e--Semola that the dimension of an $\RCD(K,N)$ space is locally constant \cite[Theorem 1.11,1.12]{brue-semola:18}. These results are summarized in the following.

\begin{Theorem}[\cite{MN, DePhilippisGigli, GigliMondinoRajala, Kapovitch-Mondino, AmbrosioHondaTewodrose, brue-semola:18}]\label{thm:MondinoNaberImproved}

Let $(X,d,\mf{m})$ be an $\RCD(K,N)$-space for some $K\in \R$ and $N\in (1,\infty)$. Then there is exactly one integer $k\in \set{1,\dots,\floor{N}}$, called the essential dimension of $X$, and a decomposition as a disjoint union $X=Z\bigcup \mc{R}^k$ such that:
\begin{enumerate}
\item $\mf{m}(Z)=0$ and $\mf{m}\mres \mc{R}^k$ is mutually absolutely continuous with $\mc{H}^k \mres \mc{R}^k$, and every point of $\mc{R}^k$ is an $\mc{H}^k$-density point,
\item  \cite[Mondino--Naber, Theorem 1.3]{MN} for any $\eps>0$ there exists an $\mf{m}$-null set $Z_\eps$ and countably many measurable sets $U^\eps_i\subset X$ such that $\mc{R}^k\subset Z_\eps \cup \bigcup_{i\in \N} U^\eps_i$ and each $U_i^\eps$ is $(1+\eps)$-biLipschitz to a subset of $\R^k$.

\end{enumerate}
\end{Theorem}

If the essential dimension
of $X$ equals $k$, then the {\em singular set} $\mc{S}$ of $X$ consists of those points admitting a tangent cone that is {\bf not} isometric to $\R^k$. Hence, following the notation of the previous theorem, $\mc{S}=Z$, and the complementary {\em regular set} $\mc{R}$ satisfies $\mc{R}=\mc{S}^{c} = \mathcal{R}^k$.
The $\eps$-regular set $\mc{R}_\eps$ consists of points admitting a ball of radius $\eps$ which is $\eps$-close in the Gromov--Hausdorff topology to a ball in $\R^k$.
Therefore, $\mc{R}_\eps$ contains the union $\bigcup_{i\in \N} U^\eps_i$ and, in particular, has full measure interior, even though $\mc{R}$ may not.

\begin{Remark}\label{rmrk-HausMeasSingSet}
As a consequence of the above theorem we always have $\mc{H}^{N}<<\mf{m}$ (i.e. $\mc{H}^{N}(Z)=0$ and $\mc{H}^{N}(Z^\eps)=0$.) Note that we may have $k<\floor{N}$. For example, for any $N>1$, $(X,d,\mf{m})=((0,\infty),\abs{\cdot},\sinh^{N-1}(x)dx)$ is an $\RCD(-(N-1),N)$ space with $k=1$.
\end{Remark}

The singular set $\mathcal{S}$ is naturally stratified
\[
\mathcal{S}^0\subset \mathcal{S}^1\subset \ldots \subset \mathcal{S}^{N-1}=\mathcal{S}.
\]
Here, $\mathcal{S}^k$ is the set of points $x\in X$ for which no tangent cone in ${\rm Tan}(X,d,\mf{m},x)$ splits off a Euclidean space $\R^{k+1}$.  The {\it boundary $\pa X$} of $X$ can then be defined in terms of stratified singular sets as $\pa X=\mc{S}^{N-1}\setminus\mc{S}^{N-2}$ (see \cite{DePhilippisGigli}).

We say that an $\RCD(K,N)$ space $(X,d,\mf{m})$ is {\em non-collapsed} if $\mf{m}=\mc{H}^{N}$, i.e. $\mf{m}$ is the $N$-dimensional Hausdorff measure. In this case, $N\in\N$ and the essential dimension of $X$ equals $N$.

By a result of Kapovitch and Mondino  \cite[Theorem 1.7]{Kapovitch-Mondino}, if $X$ is non-collapsed and $\pa X=\emptyset$ then the Hausdorff dimension of its entire singular set $\mc{S}$ is at most $N-2$.

To proceed, we now recall the standard definition of a cone of a metric measure space. First, given a metric space $Z$, the {\it cone} $C(Z)$ over $Z$ is defined as the completion of $\R^{+}\times Z$ equipped with the metric
\[
d^2_{C}((r_1, z_1), (r_2, z_2)) = \left\{
	\begin{array}{ll}
		r_1^2 + r_2^2 - 2 r_1 r_2 \cos (d_{Z}(z_1,z_2))  & \mbox{if } d_{Z}(z_1,z_2) \leq \pi \\
		(r_1 + r_2)^2 & \mbox{if } d_{Z}(z_1,z_2) \geq \pi.
	\end{array}
\right.
\]
If $(Z, d_{Z}, \mf{m}_{Z})$ is a metric measure space, then the cone $C(Z)$ admits the following {\it cone measures}
\[
\mf{m}_{C,N}=t^{N-1}\otimes \mf{m}_{Z}.
\]
Here $N>1$ is a real parameter.

The following lemma due to Kapovitch--Mondino builds upon the work of De Philippis--Gigli \cite{DePhilippisGigli-volume cone, DePhilippisGigli} and Ketterer \cite{Ketterer}.

\begin{Lemma}\cite[Lemma 4.1]{Kapovitch-Mondino}
\label{lem:ncRCD-tangents}
Let $(X,d, \mc{H}^{N})$ be a non-collapsed $\RCD(K,N)$ space.
Then, for every $x$ in $X$, every $Y\in {\rm Tan}(X,d,\mc{H}^{N},x)$ is a metric measure cone over a non-collapsed $\RCD(N-2,N-1)$ space $Z$, i.e. $Y=C(Z)$.
\end{Lemma}

We are now ready to include the following, also due to Kapovitch and Mondino:

\begin{Definition}\cite[Definition 4.2]{Kapovitch-Mondino}
Given a non-collapsed $\RCD(K,N)$ space, $K\in \R, N\in \N$, define the $\RCD$-boundary of $X$ as:
\[
\pa X\, :=\, \{ x\in X\, :\, \mbox{\rm there is}\,\, Y\in {\rm Tan}(X,d,x)\,\, \mbox{\rm such that}\,\, Y=C(Z)\,\, \mbox{\rm and}\,\, \pa Z \neq \emptyset \}.
\]
\end{Definition}

Observe that this notion is well defined, by recursively considering increasing dimensions using Lemma \ref{lem:ncRCD-tangents} above.  There is also a notion of a {\em reduced boundary}, which was shown to be a subset of $\pa X$ (Lemma 4.5 of \cite{Kapovitch-Mondino}). Recently, the reduced boundary, and some other notions of boundary, such as the one introduced after Remark \ref{rmrk-HausMeasSingSet}, were shown by Bru\`e, Naber, and Semola  \cite[Theorem 6.6]{BrueNaberSemola} to be equivalent in the case that the $\pa X$ vanishes for any non-collapsed $\RCD(K,N)$ space.

\subsection{The Coarea formula and the definition of the Jacobian matrix}\label{subsec:Jacobian}

For $N\in\N$, consider a Lipschitz map $u:\R^N\to Y$ to a metric measure space $(Y,d_Y, \mathfrak{m}_{\mathrm{Y}})$.
Kirchheim \cite{Kirchheim94} defined a seminorm on $\R^N$, called the {\em metric differential} $\op{md}(u,x)$, by
\[
\op{md}(u,x)(v):=\lim_{t\searrow 0} \frac{d_Y(u(x+tv),u(x))}{t},
\]
which exists for $\mc{H}^{N}$-a.e. point $x\in \R^N$.

Following Definition  3.25 of Reichel \cite{Reichel09},  the \emph{coarea factor} for $\op{md}(u,x)$
is defined to be
 \[
 C_N(\op{md}(u,x))=\frac{\mc{H}^{N}_{\op{md} (u,x)}(A)}{\mc{H}^{N}(A)},
 \]
 if the kernel of $\op{md}(u,x)$ is equal to $\{0\}$, otherwise $C_N(\op{md}(u,x))=0$. Here $\mc{H}^{N}_{\op{md}(u,x)}$ is the Hausdorff $N$-dimensional measure on $\R^N$ with respect to the semi-norm $\op{md}(u,x)$,  $\mc{H}^{N}$ is the standard Hausdorff measure on $\R^N$, and $A$ is any $\mc{H}^{N}$-measurable subset of positive measure. This definition is independent of the choice of $A$ (see the discussion after Definition 3.25 of \cite{Reichel09}).

 \begin{Remark}
 In \cite{Reichel09}, a more general coarea factor $C_{m}(f,x)$ is defined which is used in the statement and proof of a more general coarea formula than the one that appears below in Theorem \ref{thm:coareaReichel}.  This coarea factor agrees with the one above in the case we use, namely $m=N$.

 We also note the comment after equation (3.2) of \cite{Reichel09} that in our setting ($m=N$), $C_N(\op{md}(u,x))$ agrees with the Jacobian factor
 defined by Kirchheim \cite{Kirchheim94}.
 \end{Remark}

 Let $X$ be an $\mc{H}^{N}$-rectifiable set and $Y$ an $\mc{H}^{N}$-$\sigma$ finite metric space, for a Lipschitz map $f:X\to Y$ we define the {\em coarea factor} of $f$ at $x \in \alpha_i(U_i)$ to be
 \[
C_N(f,x):=\frac{C_N(\op{md}(f\of \alpha_i,{\alpha_i^{-1}(x)})}{C_N(\op{md} (\alpha_i),{\alpha_i^{-1}(x)})},
 \]
 where
  {$\{(U_i, \alpha_i)\}_{i \in \N}$ is a disjoint bilipschitz parametrization of $X$ as in Lemma 5.2 in  \cite{Reichel09}}. Reichel's Proposition 5.4 \cite{Reichel09} shows that $C_N(f,x)$ is a.e. independent of the parametrization. That is, $C_N(f,x)$ might be different at some points, but the points in which that happens has zero measure. (We remark that Reichel defines coarea factors $C_m$ for $m\leq N$, for use in a general coarea formula. However, we only use the case $m=N$ which simplifies to the above expression.)

Ambrosio and Kirchheim proved area (Theorem 8.2 of \cite{Ambrosio-Kirchheim}) and coarea (Theorem 9.4 of \cite{Ambrosio-Kirchheim}) formulas for countably $\mc{H}^N$ rectifiable spaces (see also \cite[Theorem 1.4]{Ka}).

While these apply in our setting, they have been generalized in a more directly applicable form in the coarea formula given by Reichel:

\begin{Theorem}[Reichel, Theorem 5.5  \cite{Reichel09}, $m=N$ case]
\label{thm:coareaReichel}
Let $X$ be an $\mathcal{H}^{N}$-rectifiable metric space.
 Suppose $N\geq 1$ and suppose $Y$ is an $\mathcal{H}^{N}-\sigma$-finite metric space. Suppose $f:X \to Y$ is a Lipschitz map and $E\subset X$ is an $\mathcal{H}^{N}$-measurable subset. Then
	\[
	\int_{E}C_{N}(f,x)d\mathcal{H}^{N}(x) = \int_{Y} \mathcal{H}^{0}(f^{-1}(y) \cap E)d \mathcal{H}^{N}(y).
	\]

	Suppose $g:X \to {\bf R}$ is an $\mathcal{H}^{N}$-integrable function. Then
	\begin{equation}
	\int_{E}g(x)C_{N}(f,x)d\mathcal{H}^{N}(x) =  \int_{Y}\int_{f^{-1}(y)\cap E} g(x) d \mathcal{H}^{0}(x)d \mathcal{H}^{N}(y).  \label{coarea-g}
	\end{equation}
\end{Theorem}

 For the rest of the section we will specialize to the case that $X$  and $Y$ are non-collapsed $\RCD(K,N)$, spaces and $Y$ is also a $\CAT(-1)$ space.
 For a Lipschitz map $f:X\to Y$ we define the Jacobian of $f$ to be
 \begin{align}\label{eq:Jac}
\Jac_x f=\limsup_{r\to 0}\frac{ \mc{H}^{N}(f(B(x,r)))}{\mc{H}^{N}(B(x,r))}.
 \end{align}
 From the definition of the Hausdorff measure, $\Jac_xf$ is clearly $L^\infty$ as a function of $x$ with global bound $\op{Lip}(f)^N$.

In Equation 8.2 of Theorem 8.1 by Ambrosio--Kirchheim \cite{Ambrosio-Kirchheim} a tangential differential for  Lipschitz maps $g:S\to Z$ from a countably $\mc{H}^N$-rectifiable space $S$ to the dual of a separable Banach space is defined.
One can easily extend this to Lipschitz maps $F:X\to M$ where $M$ is a $C^1$ Riemannian manifold by taking charts in $M$.
We will write $d_xF:T_xX\to T_{F(x)}M$ for this tangential differential, which is defined for almost every $x\in X$.

\begin{Lemma}\label{lem:Jacobian}
For a Lipschitz map $f:X\to Y$ between two non-collapsed $\RCD(K,N)$ spaces $\mc{H}^{N}$-a.e. $x\in X$, we have
 \begin{align}\label{eq:Jac_equiv}
\abs{\det d_xf}=  \Jac_x f=C_N(f,x).
 \end{align}

 \end{Lemma}

\begin{proof}
By Theorem 8.1 of \cite{Ambrosio-Kirchheim} the tangential differential $d_xf:T_xX \to T_yY$ of $f:X\to Y$ exists almost everywhere, and moreover on the regular set the resulting Banach spaces $T_xX$ and $T_{f(x)}Y$ are Hilbertian, and in particular carry their Euclidean norm.

Consequently the tangential differential is just given by a linear map on an orthonormal basis and thus $\abs{\det d_xf}$ will coincide with $\Jac_xf$ (see the comment before Lemma 4.2 of \cite{Ambrosio-Kirchheim}).

Moreover, formula \eqref{eq:Jac} above shows the second equality in formula \eqref{eq:Jac_equiv}.
This also follows from the fact that the area and coarea formulas (Theorems 8.2 and 9.4 of \cite{Ambrosio-Kirchheim}) agree in codimension $0$, i.e. in the case $m=N$.
\end{proof}

While most likely known to experts, we are not aware that the statement of \cref{lem:Jacobian} has appeared in print, even for Alexandrov spaces.

 Using the coarea formula we may deduce the following.
\begin{Lemma}[Sard's Lemma]\label{lem:sard}
Let $(X,d,\mc{H}^N)$ be an $\RCD(-(N-1),N)$ space and $Y$ be an $N$-manifold.
For any Lipschitz map $f:X\to Y$ and $E\subset X$ measurable, we have $\mc{H}^{N}(f(E))=0$ if and only if $\Jac_xf=0$ for $\mc{H}^{N}$-a.e. $x\in E$.
\end{Lemma}

\begin{proof} Apply Theorem \ref{thm:coareaReichel}
to obtain
\begin{equation}\label{Eqn:CoareaFormula}
	    	\int_{E} \Jac_xf d\mathcal{H}^{N}(x) = \int_{f(E)} \#(f^{-1}(y)\cap E)d \mathcal{H}^{N}(y).
	\end{equation}
If $\mc{H}^{N}(f(E))=0$ then the right hand side vanishes and thus on the left hand side $\Jac_x(f)$ must vanish almost everywhere on $E$.
Conversely if $\Jac_x(f)=0$ for $\mc{H}^{N}$-a.e. $x\in E$, then the left hand side of (\ref{Eqn:CoareaFormula}) above vanishes, and thus the right hand side does as well. However, this is the preimage counting function on $Y$ and hence $\mc{H}^{N}(f(E))=0$.
\end{proof}

\subsection{Structure of spaces that are both \texorpdfstring{$\CAT(\kappa)$ and $\RCD(K,N)$}{} }

Here we collect the various properties of the target space in \cref{thm:schwarz} and \cref{thm:stab} that we will need. We place these in the following lemma which combines various results from
\cite{KapovitchKetterer}, \cite{Nikolaev:89}, \cite{Otsu},  and \cite{Berestovskij-Nikolaev:93}.

\begin{lemma}\label{lem:target_space}
Suppose $(Y,d_{Y},\mf{m}_{Y})$ is a locally $\CAT(\kappa)$ and $\RCD(K,N)$ space, then $K\leq \kappa(N-1)$ and,
\begin{enumerate}
	\item $(Y,d_{Y})$ is an Alexandrov space, specifically $\CBB(K-\kappa(N-2))$,
	\item Harmonic coordinates on $Y$ form a $C^3$-structure, and $Y$ is a smooth topological manifold,
	\item The metric on $Y$ is induced from a $C^{1,\alpha}\cap W^{2,p}$-Riemannian structure for all $p\geq 1$,
	\item The distance function $d_{Y,x}(\cdot)=d_{Y}(x,\cdot)$ satisfies $$\cot_{K-\kappa(N-2)}(d_{Y,x})\geq \op{Hess}(d_{Y,x})\geq \cot_{\kappa}(d_{Y,x})$$ in the weak sense, but only up to the injectivity radius about $x$ for the lower bound,
	\item For any $\eps>0$, $K'>\kappa$, and $K''<K-\kappa(N-2)$ there is a smooth Riemannian metric $g$ on $Y$ with sectional curvatures in $[K',K'']$ such that $(Y,g)$ is $(1+\eps)$-biLipschitz homeomorphic to $(Y,d_Y)$. In particular, their respective Hausdorff measures relate by
	\[
	\frac{1}{(1+\eps)^N} \le \frac{\mc{H}_{d_Y}^N(Y)}{\Vol_g(Y)}\leq (1+\eps)^N.
	\]
\end{enumerate}
\end{lemma}

\begin{proof}
First we note that by Theorem 1.1 of Kapovitch and Ketterer \cite{KapovitchKetterer}, that the $\RCD$ constant $K$ for $Y$ satisfies $K\leq \kappa (N-1)$ and that $Y$ is an Alexandrov space of curvature bounded below by $K- \kappa (N-2)$.
In particular, $Y$ is homeomorphic to a $C^\infty$ manifold and the distance is induced from a $C^{1,\alpha}$ Riemannian metric with respect to a harmonic atlas.

By the metric Cartan--Hadamard theorem  (see Burago--Burago--Ivanov \cite{BurBurIva}), $\til{Y}$ is a globally $\CAT(\kappa)$ space.
By Theorem 3.5 of Otsu \cite{Otsu}, the theory of Jacobi fields holds a.e. on the $\CAT(\kappa)$ Alexandrov space $\til{Y}$.
As a consequence, the Hessian at $y\in Y$ of the distance function $d_{Y}(x,\cdot)$ is defined for a.e. $x,y\in Y$ and has the comparison bound \[
\Hess_x(d_{Y}(x,y))(v,v)\geq \cot_\kappa(d_{Y}(x,y)),
\]
for all $v$ orthogonal to $\grad_x d_{Y}(x,y)$ with respect to the $C^{1,\alpha}$ Riemannian metric which induces the Alexandrov metric.
(See also Kapovitch and Ketterer's Theorem 4.7 \cite{KapovitchKetterer} for a similar bound in an ostensibly more general context.)

The last statement is a restatement of the Approximation Theorem 3.1 of Nikolaev \cite{Nikolaev:89} (see also Theorem 15.1 of Berestovskij--Nikolaev \cite{Berestovskij-Nikolaev:93}).
\end{proof}

\begin{Remark}\label{rmrk-Exam}
Under the hypotheses of \cref{thm:schwarz}
 it turns out that $Y$ is homeomorphic to a closed smooth manifold by the previous lemma.
Hence the {\em orientable}
hypothesis makes sense, and can always be achieved by passing to a double cover if $Y$ is not orientable.
However, there may exist examples of such $Y$, even among negatively curved (good) orbifolds, which are not negatively curved Riemannanian manifolds \cite{DJL}.

\end{Remark}

\begin{Remark}\label{rmrk-moreExam} There are even more interesting examples of $\RCD$ spaces $X$ which satisfy the hypotheses of \cref{thm:schwarz}.
For instance, by a result of F.Galaz-Garc\'ia--Kell--Mondino--Sosa \cite[Corollary 8.10]{GalazEtAl:18}, the leaf space of an $\RCD$-space that admits a bounded metric-measure foliation (i.e. foliations with equidistant leaves of bounded diameter whose Wasserstein distance on point masses on the quotient agrees with the distance between leaves) is an $\RCD$-space.
This includes submetries, and quotients by isometric actions of compact groups on $\RCD$-spaces.
\end{Remark}

\subsection{Sobolev to Lipschitz Property  for Maps}
A metric measure space is said to satisfy the Sobolev to Lipschitz property if every Sobolev function with a uniformly bounded {minimal weak upper gradient} has a Lipschitz representative. RCD spaces are known to satisfy the Sobolev to Lipschitz property for real valued functions.
Since we will need such result for maps, we first recall the definition of Sobolev maps between RCD spaces.

{
\begin{Definition}[Sobolev map]\label{defsoble}
Let  $(X,d_{X},\mf{m}_{X})$
be a finite (dimensional), possibly non-compact,
$\RCD$ space, and
$(Y,d_{Y},\mf{m}_{Y})$
a finite (dimensional) compact $\RCD$ space.

We say that a map $F:U \to Y$ is a \textit{Sobolev map}, where $U \subset X$ is an open set, if the following two conditions hold:
\begin{enumerate}
\item For any Lipschitz function $\varphi$ on $Y$ we have $\varphi \circ F \in W^{1, 2}(U,
d_{X},\mf{m}_{X})$.
\item There exists $G \in L^2(U, \mf{m}_X)$ such that for any Lipschitz function $\varphi$ on $Y$ we have
\begin{equation}\label{eq-SobMap}
|\nabla (\varphi \circ F)|(x) \le {\rm Lip}( \varphi)\, G(x)\quad \text{for $\mf{m}_X$-a.e. $x \in U$}.
\end{equation}
\end{enumerate}
The smallest Borel function $G$ that satisfies \eqref{eq-SobMap} is denoted by $G_F$.
\end{Definition}

Using the Sobolev to Lipschitz property for functions on $\RCD$ spaces,
Honda and Sire \cite[Proposition 3.6]{Honda-Sire} showed that this property also holds for maps.

\begin{Proposition}[Sobolev to Lipschitz property for Sobolev maps] \label{lem:C-Lip}
Let $(X,d_{X},\mf{m}_{X})$ and $(Y,d_{Y},\mf{m}_{Y})$ be two compact $\RCD$ spaces and let $F:X \to Y$ be a Sobolev map and let $L \in [0, \infty)$. The following two conditions are equivalent.
\begin{enumerate}
\item The map $F$ has a Lipschitz representative with
\begin{equation*}
d_Y(F(x), F(x'))\le L d_X(x,x')
\end{equation*}
for all $x, x' \in X$.
\item We have $G_F(x)\le L$ for $\mf{m}_X$-a.e. $x \in X$.
\end{enumerate}
\end{Proposition}

Note that in the particular case when $F$ is Lipschitz $|d_xF| = G_F(x)$ a.e. $x \in X$.
}

\section{Bounds on the  index invariant of maps from \texorpdfstring{$\RCD(K,N)$}{} spaces}\label{sec:degree}

 The main result of this section is

 Theorem \ref{thm:ind-equal-homind} providing lower bounds on the average pre-image counting function in terms of the topological indices $\ind_{\pi}(f)$ and $\ind_{H}(f)$.
The result establishes relationships between our notion of  homology index in \eqref{def-h-ind}, and other fundamental ideas
used in traditional degree theory.

The well known Brouwer topological degree theory for topological manifolds (cf \cite{Lloyd}) can not be easily generalized {\it in toto} to metric spaces, unless the spaces and maps retain certain essential properties, such as a rank one top dimensional homology group.
Even in a context where a generalized topological degree theory and an analytic degree theory both make sense, connecting these together can prove challenging.
Indeed, local analytic notions of degree for Lipschitz maps can be formulated for a fairly wide class of metric spaces.
However, we do not know yet it such an analytic degree is globally pointwise constant.

In the existing proofs of the invariance of local degree, the underlying domain space must have neighborhoods of homotopy tracks\footnote{Homotopy tracks here mean the curve formed by taking the image of a point under the entire homotopy in the target space.} between oriented sets of pre-images of points induced by the map.
These neighborhoods should be absolute neighborhood retracts (ANR's).
This fundamental property
lies at the core of every proof we know of the invariance of local degrees.
This is always the case for smooth manifolds due to the local Euclidean structure.

In our case, we do not have this property on our $\RCD$ spaces. For example, there exist spaces $X$ that we consider with points having arbitrarily small neighborhoods with infinite second Betti number \cite{Menguy:00}.
Thus the question of how to extend the classical degree theory to our context remains open.
Nevertheless, we will introduce an analytic quantity for Lipschitz maps $f$, called $\pre(f)$, which we show dominates $\ind_{H}(f)$ and is sufficiently sharp to still obtain our results.

\subsection{The average number of preimages \texorpdfstring{$\pre(f)$}{}}

Suppose $X$ and $Y$ are non-collapsed $\RCD(K,N)$ spaces without boundary and finite measure, and $f:X\to Y$ is a Lipschitz map.
We define the {\em average number of preimages} of $f$ to be,
\begin{equation}\label{def:pre}
\pre(f)=\frac{1}{\mc{H}^N(Y)}\int_Y \#\set{f^{-1}(y)} d\mc{H}^N(y).
\end{equation}
Note that the function $\#\set{f^{-1}(y)}$ is measurable since $f$ is continuous.
Moreover, since $\mf{m}(X)=\mathcal{H}^{N}(X)$ is finite, the coarea formula (Theorem \ref{thm:coareaReichel}) implies that for a.e. $y\in Y$, the set $f^{-1}(y)$ is necessarily finite.
(Note that the hypotheses in Theorem \ref{thm:coareaReichel} are satisfied by our $\RCD(K,N)$ domain and target spaces.
Moreover, the coarea factor in the Jacobian vanishes on all lower dimensional strata $E_i\subset X$ for $i<N$ by definition
(see Section \ref{subsec:Jacobian}).)
By Section \ref{subsec:Jacobian}, we have $\det d_xf$ defined almost everywhere.
Moreover, the image under a Lipschitz map of the measure zero set where $\det d_xf$ is not defined has zero measure.
Hence, for almost every $y\in Y$ and every $x\in\set{f^{-1}(y)}$ we have $\det d_xf$ defined.
In particular, the {\em pointwise analytic degree} of the map $f$,
\begin{equation}\label{degree1}
\deg(f,y):=\sum\limits_{x\in f^{-1}(y)} {\rm sign} \det ( d_xf ),
\end{equation}
exists and is finite for a.e. $y\in Y$.

Now
\[
\#\set{f^{-1}(y)} \geq \abs{\sum_{x\in f^{-1}(y)} \op{sign}(\det d_x f)}.
\]
We therefore obtain
\begin{equation}\label{pre-inequality}
\pre(f)\geq \frac{1}{\mc{H}^N(Y)} \int_Y \abs{\deg(f,y)} d\mc{H}^N(y).
\end{equation}
The right hand side of (\ref{pre-inequality}) might be a more natural definition for the absolute degree, but since our $\pre(f)$ majorizes this quantity, it will turn out to be preferable.

\begin{Remark}
If $X$ were a smooth closed manifold, then it is well known that $\deg(f,y)$ is essentially constant in $y$ and a homotopy invariant of the map $f$ called the degree of $f$, $\deg(f)$.

Also, we observe that on the one hand if the Hausdorff dimension of $X$ is less than $N=\dim Y$, then $\pre(f)=0$ by Sard's Theorem (see Lemma \ref{lem:sard}).
On the other hand, when $f_*H_N(X)=\set{0}$, even if $f$ is homotopic to a constant map, it may be the case that $\pre(f)>0$.
In other words, $\pre(f)$ is only a geometric invariant, but it is bounded from below by computable topological invariants as we will see shortly.
\end{Remark}

\subsection{Bounds between the \texorpdfstring{$\ind_\pi(f)$ and $\ind_H(f)$  invariants and $\pre(f)$}{}}\label{sec:pre}

For the rest of this section we assume that $(Y,d_Y)$ is a locally $\CAT(\kappa)$ space, in addition to the assumption that $Y$ (as well as $X$) is a non-collapsed $\RCD(K,N)$ space (i.e., $(Y,d_Y,\mc{H}^{N})$, as well as $(X,d_X,\mc{H}^{N})$, is an $\RCD(K,N)$ space).
Therefore, by Lemma \ref{lem:target_space}, $Y$ is a smooth manifold. We further assume it is closed, orientable, and equipped with a $C^{1+\alpha}$-Riemannian metric.

Recall the definition of the homological index $\ind_H(f)$ and the fundamental index $\ind_\pi(f)$ from the introduction. Observe that these are nonnegative integral homotopy invariants of the map $f$.

\begin{Remark} \label{remark:fundamental and homological indeces}
We observe that the $\ind_H(f)$ may be strictly larger than $\ind_\pi(f)$. For example, let $X=M\# M$ be the connected sum of two copies of a
hyperbolic manifold $M$ of dimension $N$ and let $f:X\to M$ be the map which first collapses the connecting sphere in $X$ and then quotients under the reflection map on the resulting wedge product of $M$ with $M$.
This map is surjective between fundamental groups, but takes the fundamental class of $X$ to twice that of $M$ and hence has $\ind_H(f)=2$ but $\ind_\pi(f)=1$.
Note that the connecting sphere is nontrivial in the $(N-1)$-th homotopy group $\pi_{N-1}(X)$, so $X$ does not admit a hyperbolic (or even $\CAT(0)$) metric, unless $N=2$.
\end{Remark}

We now establish some lower bounds for $\pre(f)$.

\begin{Proposition}\label{prop:pre-pi1}
We have the following lower bound for $\pre(f)$:
\begin{equation}\label{eq:pre-bd-indpi}
    \pre(f)\geq \ind_\pi(f)
\end{equation}
Moreover, $\ind_\pi(f)$ divides $\ind_H(f)$.
\end{Proposition}
\begin{Example}
By the proposition above, if $[\pi_1(Y),f_*\pi_1(X)]=\infty$ then
\[ [H_N(Y):f_*H_N(X)]=\infty,
\]
which in turn implies $f_*H_N(X)=\set{0}$. In this case, we may have $\pre(f)=0$, such as when $f$ is a constant map. However, we may also have $\pre(f)>0$, and hence the inequality can be strict. For example, take $X=Y$ and $f:Y\to Y$ to be a map homotopic to the constant map, but with image a closed disk of $Y$ with finite preimages for each point.
\end{Example}

\begin{proof}[Proof of \cref{prop:pre-pi1}]
If $f_*H_N(X) = \set{0}$, then the right hand side of the inequality \eqref{eq:pre-bd-indpi} vanishes.
So we may assume $f_*H_N(X) \neq \set{0}$.

By
covering theory
there exists a cover $p:\hat{Y}\to Y$ with $p_*\pi_1\left( \hat{Y} \right) =f_*\pi_1(X)<\pi_1(Y)$.
By construction, this cover satisfies the lifting condition for $f$, so there is a lift $\hat{f}:X\to \hat{Y}$ such that $p\of \hat{f}=f$. Moreover, by functoriality, passing to homology, we have that  $H_N(X)\stackrel{f_*}{\longrightarrow}H_N(Y)$ is the composition of
$$H_N(X)\stackrel{\hat{f}_*}{\longrightarrow}H_N(\hat{Y})\stackrel{p_*}{\longrightarrow}H_N(Y).$$
Note that $$[\pi_1(Y),f_*\pi_1(X)]=[\pi_1(Y),p_*{\pi_1}(\hat{Y})]=[H_N(Y):p_*H_N(\hat{Y})],$$ where the last equality follows from the fact that covering maps have exactly $[\pi_1(Y),p_*{\pi_1}(\hat{Y})]$ preimages and that we may apply degree theory on the closed manifold $Y$.
Group indexes are multiplicative under composition, therefore $[\pi_1(Y),f_*\pi_1(X)]$ divides $[H_N(Y):f_*H_N(X)]$, when these are finite. (If $[\pi_1(Y),f_*\pi_1(X)]$ is infinite then so is $[H_N(Y):f_*H_N(X)]$ and $\ind_\pi(f)=\ind_H(f)=0$.)

As $Y$ is a closed manifold, the map $f$ is surjective, and hence so is $\hat{f}$. Therefore, the number of preimages of $f$ of any point $y\in Y$ is at least one for each of the $\deg(p)=[\pi_1(Y),f_*\pi_1(X)]$ preimages in $\hat{Y}$.
\end{proof}

We will need the following definitions for the next proposition. Assume that $N$ is a nonnegative integer and the metric space $Y$ has $N$-th singular homology group  $H_N(Y)\cong \Z$. For any singular homology class $\alpha\in H_N(X)$ define
\[
\norm{\alpha}_\infty:=\inf_{c\in \alpha}\sup_i\set{\abs{a_i}\in \N \, :\, c=\sum_i a_i \sigma_i,\,  \sigma_i: \Delta^N\to X  \text{ singular simplices, } a_i\in \Z}.
\]

(Note that this is not a true norm for any field since it is $\N$-valued, in general it is only a seminorm.)
Let
\[
Z(f)  =  \set{\alpha\in H_N(X)\,:\,[H_N(Y):\inner{f_*(\alpha)}]  =  [H_N(Y):f_*H_N(X)]}
\]
 be the set of singular homology classes whose image generates a subgroup of $H_N(Y)$ achieving the homological index. Set
\begin{align*}
\op{maxco}(f)=\inf_{\alpha\in Z(f)}\norm{\alpha}_\infty
\end{align*}
to be the infimum over classes in $Z(f)$ (same as the minimum in $\N$) of the $\ell^\infty$-seminorm.

\begin{Proposition}\label{prop:pre-hom}
We have the following lower bounds for $\pre(f)$:
\[
\pre(f)\geq \frac{\ind_H(f)}{\op{maxco}(f)}.
\]
\end{Proposition}
\begin{proof}
Since the statement is trivial when $\ind_H(f)=0$, we may assume
\[
[H_N(Y):f_*H_N(X)]<\infty.
\]
Consider a singular cycle $c=\sum_i a_i \sigma_i\in [c]\in Z(f)$ with singular simplices $\sigma_i: \Delta^N\to X$ and $a_i\in \Z$. Recall that by definition the class $[c]\in H_N(X)$ achieves $[f_*(c)]=k[Y]$ where
\[
k=[H_N(Y):f_*H_N(X)]<\infty,
\]
 and $[Y]\in H_N(Y)$ is the fundamental class generating $H_N(Y)$.

In particular, $[\sum_i a_i f\of \sigma_i]=k[Y]$. Since $[Y]$ has a representative singular cycle with all coefficients $1$ and whose support is all of $Y$, for any point $y\in Y$ we have
\[
k\leq \sum_{\set{j\, :\, y\in f\of \sigma_j(\Delta^N)}} \abs{a_j}\leq \#\set{f^{-1}(y)}\max_i\set{\abs{a_i}}.
\]
Average over $Y$, and take the infimum over cycles representing classes in $Z(f)$ to obtain $k\leq \pre(f)\op{maxco}(f)$, as desired.\end{proof}

\begin{Corollary}
Suppose there exists an $N$-dimensional simplicial complex $K$ which is homotopy equivalent to $X$, and such that every $(N-1)$-face of $K$ bounds at most two $N$-faces. Then, for any map $f:X\to Y$,
\[
\pre(f)\geq \ind_H(f).
\]
\end{Corollary}

To the best of our knowledge it is not currently known if $\RCD$-spaces are dominated by $CW$-complexes.
Observe that Alexandrov spaces are ANR's, and therefore, by a result of Borsuk, they are dominated by $CW$-complexes.

\begin{proof}
Let $i:K\to X$ be the homotopy equivalence. Any integral cycle $c$ with $[H_N(Y):\inner{f_*[c]}]=[H_N(Y):f_*H_N(X)]$ has a pullback $[i^*(c)]$ on $K$ which has a simplicial cycle representative $c'$. If $\sigma$ is an $N$-cell with a boundary cell adjacent only to $\sigma$, then the coefficient of $\sigma$ must be $0$ to be a cycle. Similarly
every $N$-cell $\sigma'$ adjacent to an $N$-cell $\sigma$ must carry the same coefficient with opposite sign, or else the boundary maps will fail to cancel. Hence the nonzero coefficients of $c'$ must all be constant $r$. If $r\neq \pm 1$ then $c''=\frac{1}{r} c'$ is an integral cycle with
\[
[H_N(Y):\inner{f_*i_*[c'']}]<[H_N(Y):\inner{f_*[c]}],
\]
contrary to hypothesis. It follows that $\op{maxco}(f)=1$ and by Proposition \ref{prop:pre-hom} the result follows.
\end{proof}

\begin{Remark}
In the above setting note that there may not be a nontrivial $N$-homology class on $K$ since placing coefficient $1$ on boundaryless $N$-cells may lead to a conflict in assignment of orientations. We also may be in the situation where the $N$-th Betti number of $K$ is larger than one.
\end{Remark}

In the next example we will illustrate the possible pathologies that have to be considered when working with these kinds of metric spaces.

\begin{Example}\label{ex:maxco}
Consider the space $X$ formed by removing an open disk $D$ from a closed orientable hyperbolic $N$-manifold $Y$  and attaching $k$ distinct copies $D_i$ of $D$ via the identity map from $\partial D_i\to \partial D$.
Then, on the space $X := (Y\setminus D) \bigcup_{i=1}^k D_i$, form an $N$-cycle $c$ by triangulating $X$ and placing the coefficient $1$ on each $D_i$, and $k$ on each cell of $X\setminus \bigcup_{i=1}^k D_i$, with the same orientations.
This is a simple example of an $N$-homology class that must have a coefficient larger than 1 while other coefficients are equal to 1.
Note that $c$ is a cycle with some cells of coefficient $1$ and $\norm{c}_\infty=k$.
Nevertheless, this cycle is not primitive in $H_N(X)\cong \Z^k$, but rather the sum of the natural generators formed by taking the $k$ copies of the fundamental class of $Y$ which each pass through exactly one of  the $D_i$, and thus take coefficient $0$ on the remaining disks.
If $f:X\to Y$ is the map collapsing all the $D_i$ to $D$, then $\op{maxco}(f)=1$ as the $k$ copies of the fundamental classes of $Y$ in $H_N(X)$ are carried identically onto $[Y]\in H_N(Y)$.
We are not aware of any simplicial $N$-complex (or even CW-complex) $X$ such that every set of integral $N$-cycles generating $H_N(X,\Z)$  has at least one member with one coefficient with absolute value larger than $1$.
However, the space $X$ here has branching geodesics, so it can not admit an $\RCD$-structure. Perhaps it admits a $\mathsf{CD}$-structure.

\end{Example}

Let $X$ be an $\RCD(K,N)$ space, and $Y$ a closed orientable $N$-manifold.

In what follows we denote the essential supremum of a function $g: X \to \R$ by
\[
\esssup g(x)=\inf \left\{b \in \R: \mc{H}^{N}(\{x :  g(x) > b \})=0\right\}.
\]
The next proposition provides a homological lower bound for the number of preimages of a point under the map $f$.

\begin{Proposition}\label{prop:pre-est}
If $f:X\to Y$ has $f_*H_N(X)\neq \set{0}$ then
\[
\pre(f)\geq \frac{[H_N(Y):f_*H_N(X)]}{\esssup_{x\in X} [H_N(X,X\setminus\set{x}):j_*H_N(X)]}
\]
where $j_*:H_N(X)\to H_N(X,X\setminus\set{x})$ is the map induced by inclusion on pairs. In particular, the right hand side denominator does not vanish.
\end{Proposition}

\begin{proof}
Let us first recall that an absolute local degree for a continuous map $f:V\to Q$ between an arbitrary topological space $V$ and a closed oriented $N$-manifold $Q$ can be defined.
Specifically, for $q\in Q$ we have the following maps with $\Z$ coefficients
\[
\begin{tikzcd}
H_N(V) \arrow[r,"j_*"] & H_N(V,V\setminus f^{-1}(q))\arrow[r,"f_*"] & H_N(Q,Q\setminus \set{q}) &  \arrow[l,"k_*"'] H_N(Q)
\end{tikzcd}
\]
where $j$ and $k$ are the inclusions into the relative homology groups.
(Note that $k_*$ is an isomorphism induced by the image of the fundamental class $k_*[Q]$.)

We define the absolute local degree,
\[
\abs{\deg(f,q)}=[H_N(Q):k_*^{-1}f_*j_*H_N(V)].
\]

 For a generic $q\in Q$, the preimage of a Lipschitz map is a countable discrete set.
  Hence $H_N(V,V\setminus f^{-1}(q))\cong \directsum_{v\in f^{-1}(q)} H_N(V,V\setminus\set{v})$.

To aid our exposition, we introduce the following notation for the indices of homological subgroups that we work with.
Define:
\begin{align*}
\mathcal{I}[Q,q:f_{\ast}V,v] &:= [H_N(Q,Q\setminus\set{q}):f_*H_N(V,V\setminus\set{v})]\\
\mathcal{I}[V,v:j_{\ast}^{v}V] &:= [H_N(V,V\setminus\set{v}):j^v_*H_N(V)]
\end{align*}
where $j^v_*:H_N(V)\to H_N(V,V\setminus\set{v})$ is the map induced by inclusion.

By considering each component separately and factorizing the index over the compositions of the homomorphisms we have
\begin{align*}
\abs{\deg(f,q)}&\leq \sum_{\substack{v\in f^{-1}(q)\\ j^v_*H_N(V)\neq \set{0}}}
\mathcal{I}[Q,q:f_{\ast}V,v] \cdot \mathcal{I}[V,v:j_{\ast}^{v}V]\\
&\leq\#f^{-1}(q)\cdot\sup_{v\in f^{-1}(q)} \mathcal{I}[Q,q:f_{\ast}V,v] \cdot \mathcal{I}[V,v:j_{\ast}^{v}V].
\end{align*}

Note that  $[H_N(Q):f_*H_N(V)]\leq [H_N(Q):k_*^{-1}f_*j_*H_N(V)]=\abs{\deg(f,q)}$ for every choice of $q$.
Therefore we obtain:
\[
[H_N(Q):f_*H_N(V)]\leq\#f^{-1}(q)\cdot\sup_{v\in V} \mathcal{I}[Q,q:f_{\ast}V,v] \cdot \mathcal{I}[V,v:j_{\ast}^{v}V]
\]

Now we specialize to the case when $V=X$, $Q=Y$ and $f:X\to Y$ is our initial Lipschitz map.
By \cite[Theorem 4.11]{Kapovitch-Mondino}, there is a set $A\subset X$ of Hausdorff codimension at least $2$, such that $X\setminus A$ is a $C^{1+\alpha}$ manifold.
(We observe that $A$ may not contain all of the singular set $\mc{S}$ as singular points may be manifold points and can be dense.)
By the generalized Sard's Lemma~\ref{lem:sard}, the set $f(\mc{S}\bigcup A)\subset Y$ also has measure $0$.
Hence at each regular point $x\in X\setminus (\mc{S}\bigcup A)$, the tangent space is isomorphic to a Euclidean space.
Therefore $s$ has a neighborhood homeomorphic to a Euclidean disk, which implies
$[H_N(Y,Y\setminus\set{q}):f_*H_N(X,X\setminus\set{x})]=1$, because we have natural isomorphisms $H_N(Y,Y\setminus\set{q})\cong H_N(B(q,\eps),\partial B(q,\eps))$ for any $\eps>0$.
Hence, at $\mf{m}$-a.e. point $q\in Y$ we have,
\[
[H_N(Y):f_*H_N(X)]\leq \#f^{-1}(q)\cdot\sup_{x\in f^{-1}(q)} [H_N(X,X\setminus\set{x}):j^x_*H_N(X)].
\]
Average over $Y$ and apply H\"{o}lder's inequality to the right hand side to find,
\begin{align*}
[H_N(Y):f_*H_N(X)]&\leq \pre(f)\cdot\esssup_{q\in Y}\sup_{x\in f^{-1}(q)} [H_N(X,X\setminus\set{x}):j^x_*H_N(X)]\\
&\leq \pre(f)\cdot\esssup_{x\in X} [H_N(X,X\setminus\set{x}):j^x_*H_N(X)]
\end{align*}
as promised.
\end{proof}

It is unknown to the authors whether or not there exists a map
$f:X\to Y$ with $f_*H_N(X)=0$, for instance homotopic to a constant, from a noncollapsed boundary-less $\RCD(K,N)$-spaces $X$ with zero volume entropy, $h(\bar{X})=0$ to a closed hyperbolic $N$-manifolds $Y$ with positive average local degree, i.e. $\pre(f)>0$, even though the absolute topological degree could be zero. (There are no examples where $X$ is a Riemannian manifold with this property.) In this case, the inequality \eqref{eq:Thm1-ineq}  in Theorem \ref{thm:entrig} would fail. While the proofs use the analytic formula for $|\deg f|$ arising from the coarea formula, these also rely on the equivariance of the lifted mapping which we do not have in the inessential case, which is why we need to make the exception for $|\deg f|=0$ in that case. Note that inequality \eqref{eq:Thm1-ineq}  is automatically satisfied when $\pre(f)=0$.

If $f_*H_N(X)=0$, we note that $f$ is homotopic to a map with image on a lower-dimensional set in $Y$. Thus the inequality on the right must be $0$ to hold (even though the first formula defined above for $\pre(f)$ may not be $0$), thus we must set it to be $0$.

\begin{Theorem}\label{thm:ind-equal-homind}
Let $X$ and $Y$
be non-collapsed $\RCD(K, N)$ spaces without boundary and $(Y,d_Y)$ a locally $\CAT(\kappa)$ space. If $f:X\to Y$ is a Lipschitz map, then
\[
\pre(f)\geq \ind_H(f).
\]
\end{Theorem}

\begin{proof}
Under the assumptions we have
\[
\esssup_{x\in X} [H_N(X,X\setminus\set{x}):j_*H_N(X)]=1,
\]
because almost every point $x\in X$ belongs to an open manifold subset and hence has an open neighborhood homeomorphic to $\R^N$. Hence the result follows from Proposition \ref{prop:pre-est}.
\end{proof}

\section{Properties of the barycenter and
natural maps}\label{sec:Proofs}
In this section we first recall the definition of  barycenter and the natural map induced as in Sambusetti \cite{Sambusetti-1999} and Sturm in \cite{Sturm03},  and establish some basic properties of the natural map in our current setting. Then we show the important property (\cref{lem:Lips}) that the natural map is Lipschitz. Finally we prove the inequality cases in Theorems \ref{thm:entrig} and \ref{thm:schwarz} assuming \cref{prop:Jac_est}.

\subsection{Barycenters}\label{sec:barycenters}
Let $\mc{P}(Y)$ be the space of probability measures on a complete metric space $Y$. Let $\mc{P}_0(Y)$ be the space of probability measure on $Y$ of the form $\sum_{i=1}^k a_i \delta_{x_i}$, i.e. finite sums of Dirac measures. We let $\mc{P}^\infty(Y)$ be the space of measures of bounded support, and for $p\in [1,\infty)$ we let $\mc{P}^p(Y)$ be the space of probability measures $\mu$ such that $d(y,\cdot)\in L^p(\mu)$ for some (hence any) $y\in Y$. We clearly have for any $\infty\geq p>q\geq 1$,
\[
\mc{P}_0(Y)\subset \mc{P}^p(Y)\subset \mc{P}^q(Y)\subset \mc{P}(Y).
\]
Moreover, $\mc{P}_0(Y)$ is dense in $\mc{P}^p(Y)$ and $\mc{P}(Y)$.
For $p\in [1,\infty]$, we equip $\mc{P}^p(Y)$ with the $L^p$ Wasserstein distance.

Let $Z$ be a complete $\CAT(0)$ space, and choose any fixed basepoint $o\in Z$.
For a measure $\nu\in\mc{P}^1(Z)$, consider the function $\mc{B}_{\nu}:Z\to \R$
given by
\begin{equation}  \label{Busemann like function}
\mc{B}_{\nu}(z)=\int_{Z}  d(y,z)^2-d(o,y)^2 d\nu(y).
\end{equation}

Note that the above is the $d^2$-barycenter used by Sambusetti \cite{Sambusetti-1999}, and also used by Sturm in \cite{Sturm03}. This will be important for some of the arguments later on.

\begin{lemma}[Proposition 4.3 of \cite{Sturm03}]\label{lem:bary_unique}
Let $(Z, d)$ be a complete $\CAT(0)$ space and fix $o \in Z$. For each $\nu\in \mathcal{P}^1(Z)$ there exists a unique point $z \in Z$ which minimizes the uniformly convex, continuous function $\mc{B}_\nu$. This point is independent of the basepoint $o$; it is called the barycenter (or, more precisely, $d^2$-barycenter) of $\nu$ and denoted by $\bary(\nu)$.
Moreover, for $\nu\in \mc{P}^2(Z)$, the following base-point free formulation holds:
\[
\bary(\nu)=\argmin_z \int_Z d(y,z)^2 d\nu(y).
\]
\end{lemma}

\subsection{The natural maps \texorpdfstring{$F_s$}{}}\label{sec:natural}

Given $f: X \to Y$ as in \cref{thm:entrig} and \cref{thm:schwarz}, using the barycenter we construct maps $F_s: X \to Y$, called natural maps, that are homotopic to $f$.

Observe that if $\ind_H(f)=0$ then the right hand side of \eqref{eq:Thm1-ineq} and \eqref{eq:Thm2-ineq} are $0$ and hence the conclusions of the corresponding theorems trivially hold. Thus we henceforth assume that $\ind_H(f)\neq 0$, i.e. that $f$ is essential.

As a first step, we replace the map $f$ with a homotopic Lipschitz map which we again call $f$. To do this, consider the smooth manifold $Y$ equipped with its $C^{1,\alpha}$-Riemannian metric. By Nash's Embedding Theorem this can be isometrically embedded as a submanifold of $\R^m$ \cite{Nash}.
By the tubular neighborhood theorem, obtained for example by integrating the normal bundle for a sufficiently small time, there is an open tubular neighborhood $U\subset \R^m$ of $Y$ with smooth boundary which admits a Lipschitz retract to $Y$, say by averaging local projections via a partition of unity. By Proposition 6.5.2 of \cite{CobzasMiculescuNicolae}, $f$ is homotopic to a Lipschitz map $\hat{f}:X\to Y$.
Now for the remainder of the proof we rename $\hat{f}$ as $f$. Since $\ind_H$ and $\ind_\pi$ are homotopy invariants, this replacement has no effect on the inequalities \eqref{eq:Thm1-ineq} and \eqref{eq:Thm2-ineq} or equality \eqref{eq:Thm2-equal}.

Let $\til{X}$ and $\til{Y}$ be the universal covers of $X$ and $Y$, respectively. Let
\[
\bar{X}=\til{X}/\ker f_*  \qquad \text{and} \qquad \Ga=\pi_1(X)/\ker f_*
\]
as before.
Let $\til{f}: \bar{X}\to \til{Y}$ be the corresponding lift of $f$ with image $\til{Y}$.
The measure $\mf{m}$ on $X$ lifts to a $\pi_1(X)$-invariant measure $\til{\mf{m}}$ on $\til{X}$ and a $\Ga$-invariant measure $\tmX$ on $\bar{X}$.
Here we are using the (nontrivial) previously indicated fact that the universal covering of $X$ exists and has an $\RCD(K,N)$ structure \cite{Mondino-Wei}.
Observe that for any fundamental domain $D\subset \bar{X}$ the measure is uniquely specified by  ${\tmX}(A)=\sum_{\ga\in \Ga}\mf{m}(\pi(A\cap \ga D))$, with $\pi:\bar{X}\to X$ the covering map.
As this lift is canonical, we do not need to specify which basepoint is used.

For each $s>h(\bar{X})$ and $x\in \bar{X}$, consider the finite measure $\mu_x^s$ supported on $\bar{X}$ absolutely continuous with respect to the measure ${\tmX}$ and with Radon--Nikodym derivative
\begin{equation}\label{eq:muxs}
\frac{d\mu_x^s}{d{\tmX}}(z)=e^{-s d(x,z)},
\end{equation}
where $d(\cdot,\cdot)$ is the distance on $\bar{X}$.

Note that the measure $\mu_x^s$ has finite total mass by the condition that $s>h(\bar{X})$.

\begin{Definition}
Set $\sigma_{x}^{s}=\til{f}_*\mu_x^s$.  We define the map $\til{F_s}:\bar{X}\to \til{Y}$ by setting
\[
\til{F_s}(x)=\bary({\sigma}_{x}^{s}).
\]
\end{Definition}

The next two lemmas are now standard, but we include these for completeness in our setting.
\begin{lemma}\label{lem:bar-pi1-equi}
The map {$\bary:\mc{P}^1(\til Y) \to \til Y$} is $\op{Isom}\left( \til{Y}\right) $ equivariant, and $\til{F_s}$ is equivariant with respect to the homomorphism $\rho:\Ga \to \pi_1(Y)$ induced by $f$.
\end{lemma}

\begin{proof}
To verify the equivariance of $\bary$, we check that for $\ga\in \op{Isom}\left( \til{Y}\right)$ and any measure $\nu\in\mc{P}^1\left( \til{Y}\right)$,
\begin{align*}
\bary(\ga_*\nu)
&=\argmin_y\int_{\til{Y}} d(y,z)^2-d(o,z)^2 d\ga_*\nu(z)\\
&=\argmin_y\int_{\til{Y}} d(y,z)^2-d(o,z)^2 d\nu(\ga^{-1}z)\\
&=\argmin_y\int_{\til{Y}} d(y,\ga z)^2-d(o,\ga z)^2 d\nu(z)\\
&=\argmin_y\int_{\til{Y}} d(\ga^{-1}y,z)^2-d(\ga^{-1}o,z)^2 d\nu(z)\\
&=\ga \argmin_y\int_{\til{Y}} d(y,z)^2-d(\ga^{-1}o,z)^2 d\nu(z)\\
&=\ga \bary(\nu).
\end{align*}
The last line follows from the independence of $\bary$ on the choice of basepoint $o$.

For the second statement, observe that for $\ga\in \Ga$, we have
\begin{align*}
d{\sigma}_{\ga x}^s(\til{f}(y))
=& \sum_{z\in \til{f}^{-1}(y)}\til{f}_*e^{-s d(\ga x,z)}d{\tmX}(z) \\
=& \sum_{z\in \til{f}^{-1}(y)}\til{f}_*e^{-s d(x,\ga^{-1}z)}d\ga_*{\tmX}(y) \\
=&  \til{f}_*\ga_*\sum_{z\in \til{f}^{-1}(y)}e^{-s d(x,z)}d{\tmX}(z)
= \rho(\ga)_*d{\sigma}_{x}^s(\til{f}(y)).
\end{align*}
Here we have used that $\ga_*{\tmX}={\tmX}$.
\end{proof}

\begin{Definition}\label{def:Fs}
   For each $s>h(\bar{X})$, we denote by $F_s:X\to Y$ the continuous map naturally given by the  equivariance of $\til{F}_s$ under the actions of  $\Ga$ on $\bar{X}$ and  $\rho(\Ga)$ on $\til{Y}$.
\end{Definition}

\begin{lemma}\label{lem:psitilde}
The map $\til{\Psi}:[0,1]\times \bar{X}\to\til{Y}$ given by
\[
\til{\Psi}_t(x)=\bary\left(t\delta_{\til{f}(x)}+(1-t){\sigma}_{x}^{s}\right)
\]
produces an explicit equivariant homotopy from $\til{F_s}=\til{\Psi}_0$ to $\til{f}=\til{\Psi}_1$. The map  $\til{\Psi}$ descends
to a homotopy $\Psi$ from $F_s$ to $f$.
\end{lemma}

\begin{proof}
Let $\rho:\Ga\to \pi_1(Y)$ be the homomorphism induced by $f$.
Since {by the previous lemma} $\mu_{\ga x}=\ga_*\mu_x$ and  $\delta_{\ga x}=\ga_* \delta_x$ for all $x\in X$ and $\ga\in \Ga$, we may verify that
\begin{align*}
\Psi_t(\ga x)
&=\bary\left(t\delta_{\til{f}(\ga x)}+(1-t){\sigma}_{\ga x}^s\right)\\
&=\bary\left(t\delta_{\rho(\ga)\til{f}(x)}+(1-t)\til{f}_*\ga_*\mu_{x}^s\right)\\
&=\bary\left(t\rho(\ga)_*\delta_{\til{f}(x)}+(1-t)\rho(\ga)_*\til{f}_*\mu_{x}^s\right)\\
&=\bary\left(\rho(\ga)_* \left(t\delta_{\til{f}(x)}+(1-t)\til{f}_*\mu_{x}^s\right)\right)\\
&=\rho(\ga)\bary\left( t\delta_{\til{f}(x)}+(1-t)\til{f}_*\mu_{x}^s\right)\\
&=\rho(\ga)\Psi_t(x).
\end{align*}

It remains to show that $\til\Psi$ is a homotopy. Observe that $\til{\Psi_t}$ is continuous in $t$ because $\bary$ is continuous with respect to the topology on finite measures induced by the Wasserstein distance, for which
\[
t\delta_{\til{f}(x)}+(1-t)\til{f}_*\mu_{x}^s
\]
is continuous in both $t$ and $x$. Moreover, as $\bary(\delta_{y})=y$ we have $\til{\Psi}_1=\til{f}$ and $\til{\Psi}_0=\til{F}_s$ by definition.
\end{proof}

\subsection{\texorpdfstring{$F_s$ is Lipschitz
}{}}\label{subsec:Fs-Lip}

For this section we will need further information on the regularity of $\til Y$.
For any manifold with two sided curvature bounds in the sense of Alexandrov, the metric is given by a $C^{1,\alpha}$ Riemannian metric with respect to an atlas of $C^{3,\alpha}$ harmonic coordinates.
On the one hand it is not known if this can be improved to a $C^{1,1}$ Riemannian metric in some coordinate chart (\cite[Problem 1.10]{KapovitchLytchak21}).
Indeed, there are counterexamples to the existence of metrics with this regularity with respect to harmonic coordinates \cite{Peters87}.
On the other hand, by Theorem 1.8 of Kapovitch and Lytchak \cite{KapovitchLytchak21}, the function $d(x,\cdot)$ is $C^{1,1}$ on $\til Y\setminus \{x\}$ for any fixed $x\in \til Y$.

\begin{Lemma}[Theorem 6.3 \cite{Sturm03}]\label{lem:barLip}
If $Z$ is a complete $\CAT(0)$ space, then the map $\bary:\mc{P}^p(Z)\to Z$ is $1$-Lipschitz for any $p\in [1,\infty]$.
\end{Lemma}

\begin{Lemma}\label{lem:Lips}
The maps $F_s$ is Lipschitz for each $s> h(\bar{X})$.
\end{Lemma}

\begin{proof}
We first show that the embedding
\[
\bar{X} \to \mc{P}^1(\bar{X}),\quad
 x\mapsto \frac{\mu_x^s}{\norm{\mu_x^s}},
 \]
 is locally Lipschitz with respect to the Wasserstein distance $W_1$.
From the definition \eqref{eq:muxs} of $\mu_x^s$ and $\mu_y^s$ in terms of ${\tmX}$,
we may estimate the Wasserstein distance $W_1$ in the following way,
\beq
& W_1\left(\frac{\mu_x^s}{\norm{\mu_x^s}},\frac{\mu_y^s}{\norm{\mu_y^s}}\right)=W_1\left(\frac{1}{\norm{\mu_x^s}} e^{-s d(x,\cdot)}{\tmX},\frac{1}{\norm{\mu_y^s}}e^{-s d(y,\cdot)}{\tmX}\right)\\
&=\sup_g \set{\int_{\bar{X}} g(z) \left(\frac{e^{-s d(x,z)}}{\norm{\mu_x^s}}-\frac{e^{-s d(y,z)}}{\norm{\mu_y^s}}\right)d{\tmX}(z) \,:\, g\in \op{Lip}(\bar{X},\R)\text{ with } \op{Lip}(g)\leq 1 }.
\eeq
Since $1=\int_{\bar{X}}\frac{e^{-s d(x,z)}}{\norm{\mu_x^s}}d{\tmX}(z)$ for any $x\in \bar{X}$, we have
\begin{eqnarray*}
\lefteqn{ \int_{\bar{X}} g(z) \left(\frac{e^{-s d(x,z)}}{\norm{\mu_x^s}}-\frac{e^{-s d(y,z)}}{\norm{\mu_y^s}}\right)d{\tmX}(z) } \\
& = &  \int_{\bar{X}} (g(z) - g(y)) \left(\frac{e^{-s d(x,z)}}{\norm{\mu_x^s}}-\frac{e^{-s d(y,z)}}{\norm{\mu_y^s}}\right)d{\tmX}(z) \\
& \leq & \int_{\bar{X}} d(z,y)\left(e^{s d(x,y)}\frac{}{}\frac{\norm{\mu_y^s}}{\norm{\mu_x^s}}-1\right)\frac{d\mu_y^s(z)}{\norm{\mu_y^s}}.
\end{eqnarray*}
 In the last step we used the triangle inequality $d(x,z)\geq d(y,z)-d(x,y)$ to obtain $e^{-s d(x,z)}\leq e^{-sd(y,z)}e^{s d(x,y)}$, and the assumption that $g$ is a $1$-Lipschitz function.

Note that by using triangle inequality in the density we have
\[
e^{-sd(x,y)}\leq\frac{\norm{\mu_y^s}}{\norm{\mu_x^s}}\leq e^{s d(x,y)}.
\]

Observe that for $y$ in a compact fundamental domain, and hence any $y\in X$, there is a positive lower bound for $\norm{\mu_y^s}$ independent of $y$ and $s_0>s\geq h$ for any fixed $h<s_0<\infty$. Moreover, by construction of the measure $\mu_y^s$, the function $r\mapsto \int_{\bar{X}\setminus B(y,r)} {\mu_y^s}$ eventually decays exponentially at infinity at least as fast as $Ce^{(h+\eps-s)r}$ for any $s-h>\eps>0$. Hence
\[
\int_{\bar{X}} d(z,y)\frac{d\mu_y^s(z)}{\norm{\mu_y^s}}\leq \frac{1}{\norm{\mu_y^s}}\sum_{i=1}^\infty i\int_{B(y,i)\setminus B(y,i-1)}d\mu_y^s\leq \frac{C}{\norm{\mu_y^s}}\sum_{i=1}^\infty ie^{(h+\eps-s)(i-1)}\leq  C_{s}'
\] for some constant $C_{s}'$ depending only on $s>h$. Note that $C_s'$ will tends to infinity as $s\to h$ in the case that the support of the probability measure $\frac{\mu_y^s}{\norm{\mu_y^s}}$ tends to infinity.

Hence for $d(x,y)<1$, there is another constant $C_s$ independent of $x$ and $y$ such that,
\begin{eqnarray*}
\int_{\bar{X}} g(z) \left(\frac{e^{-s d(x,z)}}{\norm{\mu_x^s}}-\frac{e^{-s d(y,z)}}{\norm{\mu_y^s}}\right)d{\tmX}(z) & \leq & C_s' (e^{2s d(x,y)}-1)\\
& \leq & C_s d(x,y).
\end{eqnarray*}
Therefore $W_1\left(\frac{\mu_x^s}{\norm{\mu_x^s}},\frac{\mu_y^s}{\norm{\mu_y^s}}\right)\leq C_s d(x,y)$, which is to say that the embedding is locally Lipschitz.

Now note that change of variables in the Kantorovich formula for the Wasserstein distance gives for the push-forward measures by a $C$-Lipschitz map $\til f:\bar{X}\to\til{Y}$,
\begin{equation}
W_1(\til f_*\mu,\til f_*\nu)=\sup_g \set{\int_{\bar{X}} g\of \til f(z) d(\mu-\nu)(z)\,:\, \op{Lip}(g)\leq 1 }\leq C W_1(\mu,\nu),  \label{push-lip}
\end{equation}
since $g\of \til f$ is $C$-Lipschitz.

Our measure $\mu^s_x$ is a smooth function in the distance times ${\tmX}$, and moreover $f$ is essential so $\til{f}$ is surjective.
By \cref{lem:barLip}, $\bary$ is $1$-Lipschitz.
By \cref{push-lip}, the embedding followed by push-forward of measures is also locally Lipschitz with respect to the Wasserstein distance for each $h (\bar{X})< s$.
So the composition $\til{F}_s(x)=\bary(\til{f}_*\mu_x^s)$ is locally Lipschitz for each $h (\bar{X})< s$.
Moreover, $\til{F}_s$ is equivariant with respect to both cocompact actions of $\Ga$ and $\pi_1(Y)$. Therefore, $\til{F}_s$ is globally Lipschitz for each $h(\bar{X})<s$, as stated.
\end{proof}

\subsection{Proof of inequality cases in Theorems \ref{thm:entrig} and \ref{thm:schwarz}}\label{sec:inequalities_proof}

The key to prove the inequalities is the following
  global estimate which generalizes the one originally obtained by Besson--Courtois--Gallot \cite{Besson-Courtois-Gallot:95} to the $\RCD(K,N)$ setting.
Recall the definition of the Jacobian introduced in \cref{eq:Jac} above, and that by \cref{lem:Jacobian} it is an $L^\infty$-function.

\begin{Proposition}\label{prop:Jac_est}
For all $s>h(\bar{X})$, the natural map $F_s:X\to Y$
satisfies the following inequality in the case $Y$ is negatively curved locally symmetric,
\[
\Jac F_s(x)\leq \left(\frac{s}{h(\til{Y})}\right)^{N},
\]
and
\[
\Jac F_s(x)\leq \left(\frac{s}{N-1}\right)^{N},
\]
in the case that $Y$ is as in \cref{thm:schwarz}.
\end{Proposition}
We will defer the proof of this proposition to \cref{sec:proof4.8}.

With \cref{prop:Jac_est}
the inequalities (\ref{eq:Thm1-ineq}) and (\ref{eq:Thm2-ineq}) follow quickly from the coarea formula (\ref{coarea-g}) and Theorem \ref{thm:ind-equal-homind}.

\begin{proof}[Proof of (\ref{eq:Thm1-ineq}) and (\ref{eq:Thm2-ineq})]
As $X$ is a noncollapsing $\RCD(K,N)$ space, it is rectifiable with respect to $\mc{H}^{N}$.
The space $Y$ is a locally symmetric space as in the setting of Lemma \ref{lem:target_space}, and hence $Y$ is a smooth manifold with a $C^{1,1}$-Riemannian metric.
Hence, applying the coarea formula (\ref{coarea-g})
to the case where $g=F_s$ yields (denoting $\mf{m}=\mc{H}^{N}$ for short),
\begin{align*}
\int_X \Jac F_s(x) d\mf{m}(x)
&=\int_Y \int_{F_s^{-1}(z)} d\mc{H}^0(x)d\mc{H}^N(z) \\
&=\int_Y p(z)d\mc{H}^N(z) = \pre(F_s)\Vol(Y).
\end{align*}
Here,
$p(z)=\#\set{F_s^{-1}(z)}$ is the preimage counting function which coincides with the $0$-dimensional Hausdorff measure, and we used the definition of $\pre(F_s)$, \eqref{def:pre}, in the last equality.

 Combining the above with the first inequality of Proposition \ref{prop:Jac_est} gives
\begin{align}
\pre(F_s)\Vol(Y)
 = \int_X \Jac F_s(x) d\mf{m}(x)
\leq \left(\frac{s}{h(\til{Y})}\right)^{N} \mf{m}(X).  \label{pre-Jac}
\end{align}
As this holds for all $s> h(\bar{X})$, we obtain the following by letting $s \to h(\bar{X})$,
\begin{equation}\label{eq:fullIneq}
\pre(F_s)\cdot h(\til{Y})^{N} \Vol(Y) \leq h(\bar{X})^{N}\mf{m}(X).
\end{equation}
Now (\ref{eq:Thm1-ineq}) follows from
Theorem \ref{thm:ind-equal-homind}, which shows
\[
\pre(F_s)\geq \ind_H(F_s)=\ind_H(f).
\]

Applying the second inequality of Proposition \ref{prop:Jac_est} to the equality in (\ref{pre-Jac}),
we obtain
\[
\pre(F_s)\Vol(Y)\leq \left(\frac{s}{N-1}\right)^{N} \mf{m}(X).
\]
Recalling that $s> h(\bar{X})$
and that the assumption that $X$ is $\RCD(-(N-1),N)$ implies by Theorem 3.1 in \cite{stormmod} that $h(\bar{X})\leq N-1$, again combining \cref{thm:ind-equal-homind}
 yields the inequality (\ref{eq:Thm2-ineq}). \end{proof}

\section{Rigidity Cases} \label{sec:vol-rigid}

To obtain the rigidity results in the equality cases in Theorems \ref{thm:entrig} and \ref{thm:schwarz} we critically rely on the following.
\begin{Proposition}
    \label{prop:F_converge}
In the equality case of \cref{thm:entrig} and \cref{thm:schwarz} and when $\ind_\pi(f)=1$, there exists a decreasing sequence $s_i\to h(\bar{X})$ such that  $F_{s_i}$ converges to a 1-Lipschitz map $F:X\to Y$  homotopic to $f$.
\end{Proposition}
We will defer the proof of this proposition to  \cref{sec:FinW1infty} as it is a bit long and involved.
In this section,
we first prove Theorem \ref{thm:LV-rigid} using the approach of Li and Wang \cite{LiWang-limits}.
With this and \cref{prop:F_converge}, we  prove the rigidity statements in the equality cases in Theorems \ref{thm:entrig} and \ref{thm:schwarz}.

\subsection{Proof of Theorem \ref{thm:LV-rigid}}

For the proof of Theorem \ref{thm:LV-rigid} we use the approach of Li and Wang in \cite{LiWang-limits}.

The proof of Li and Wang for non-collapsed Ricci limit spaces uses volume convergence and an almost maximal volume theorem, both of these results have been extended to $\RCD (K, N)$ spaces in \cite[Theorems 1.3, 1.5]{DePhilippisGigli} for the Hausdorff measure $\mathcal{H}^N$. It also uses that the singular set has Hausdorff dimension $\le N-2$, this is extended to noncollapsing $\RCD(K,N)$ spaces with no boundary in \cite[Theorem 1.7]{Kapovitch-Mondino}.

In addition they also use Lemma 3.1 from Cheeger and Colding \cite{Cheeger-Colding}. We state the corresponding lemma in our setting. The proof follows verbatim from theirs given that the Bishop--Gromov comparison holds in our setting.

\begin{lemma}
For all $K \in \R$, $d>0$, $\epsilon>0$
and $N$ a non-negative integer
there exists $c(N, K, d, \epsilon)>0,$ such that
the following holds. Let $(X,d,\mf{m})$ be a non-collapsed $\RCD(K,N)$ space with $\pa X=\emptyset$ and
\[
B_{\epsilon}\left(x_{1}\right) \cup \Omega \subset \overline{B_{d}(p) \backslash E}
\]
where
\[
E=\bigcup_{j} \overline{B_{r_{j}}\left(q_{j}\right)}
\]
for a countable family of balls $\set{B_{r_j}(q_j)}_{j\in\N}$,
 for some $p\in X$, and a Borel subset $\Omega\subset X$.
Then, if every minimal geodesic $\gamma :[0, \ell] \rightarrow X$ with $\gamma(0)=x_{1}$ and $ \gamma(\ell) \in$
$\Omega$ intersects $E$, we have
\[
c(N, K, d, \epsilon)\mf{m}(\Omega) \leq \sum_{j} r_{j}^{-1} \mf{m}\left(B_{r_{j}}\left(q_{j}\right)\right).
\]
\end{lemma}

We remark that morally this lemma states that if we have a set $E$ consisting of a countable union of balls and a point $x_1$ at distance $\eps$ from this set, then any ``shadow''
in $X$ as seen from $x_1$ (that is, the union of endpoints of a Borel family of geodesic segments starting at $x_1$ and passing through $E$) has the given uniform bound on its measure.

The above lemma immediately yields the following corollary.

\begin{Corollary}[Dimension Comparison]\label{cor:dimcomp}
Let $p\in X$ and $\Omega\subset X$ with $\mf{m}(\Omega)>0$ and let $E$ consist of one point on each geodesic $[p,x]$ with $x\in \Omega$. If $d(p,E)>0$, then
\[
\dim_H(E)\geq N-1.
\]
\end{Corollary}
The following version of Theorem A from \cite{LiWang-limits} holds with the same proof after replacing their Lemma 1.6 with Corollary \ref{cor:dimcomp} above.

\begin{Theorem}[Lipschitz volume Rigidity (Theorem \ref{thm:LV-rigid})]
Let $X$ and $Y$ be as in Theorem \ref{thm:entrig}. Suppose there is a 1-Lipschitz map $f : X \rightarrow Y$ with $\operatorname{vol}(X)=\operatorname{vol}(f(X)),$ then $f$ is an isometry with respect to the intrinsic metrics of $X$ and $f(X).$ In particular, if $f$ is also onto, then $X$ is isometric to $Y$ .
\end{Theorem}

\begin{Remark}
The path-isometric map from $[0,2\pi)$ to the unit circle is a volume preserving bijection which is not an isometry. Generalizations of this example are why we exclude ``boundary'' in our assumptions. More generally, there are examples of volume preserving 1-Lipschitz maps which are also bi-Lipschitz homeomorphisms, but not isometries (see  \cite[Example 1.5]{Li-AlexandrovRigidity}).
\end{Remark}

\subsection{Equality case in Theorems \ref{thm:entrig} and \ref{thm:schwarz}}\label{sec:equality_proofS}
First suppose that $\ind_\pi(f)=1$. By \cref{prop:F_converge}, in the equality cases we obtain a 1-Lipschitz map $F$ homotopic to $f$.

  Then by \cref{prop:pre-pi1} we have $1=\ind_\pi(F)\leq \pre(F)$

  and so for a.e. $y\in Y$ we have $\#(F^{-1}(y))\geq 1$. By the equality and the coarea formula,

\[
\mc{H}^N(X)=\int_{X} \Jac_x F d\mathcal{H}^{N}(x) = \int_{Y} \#(F^{-1}(y))d \mathcal{H}^{N}(y)\geq\mc{H}^N(Y)=\mc{H}^N(X).
\]
Hence we have equality everywhere and $\#(F^{-1}(y))=1$ for a.e. $y\in Y$.
It follows that for any measurable $E\subset Y$,
\[
	\mc{H}^N(F^{-1}(E))=\int_{F^{-1}(E)} \Jac_x F d\mathcal{H}^{N}(x) = \int_{E} \#(F^{-1}(y))d \mathcal{H}^{N}(y)=\mc{H}^N(E).
\]
In other words $F$ is a volume preserving map.

We are now ready to prove the equality cases of our main theorems. Note that we cannot have equality when $\ind_\pi(f)=0$, so it follows that $\ind_\pi(f)=[\pi_1(Y),f_*\pi_1(X)]<\infty$.

\begin{proof}[Proof of equality case in Theorems \ref{thm:entrig} and \ref{thm:schwarz}]

When assuming $\ind_\pi(f)=1$ and the equality case of Theorem \ref{thm:schwarz},
then it follows that $F$ is volume preserving, recalling that $f$ is homotopic to $F$.
By  \cref{prop:F_converge},  we know that $F:X\to Y$ can be taken to be $1$-Lipschitz.
Hence, by Theorem \ref{thm:LV-rigid} we conclude that $F$ is an isometry.

Now, for the equality case of Theorem \ref{thm:entrig}, we again first assume that $\ind_\pi(f)=1$.
The quantities on each side of the equality are metrically scale invariant. Normalize the metric on $Y$ so that  its sectional curvatures satisfy $K_Y\leq -1$, and the metric on $X$ so that $h(\bar{X})=h(\til{Y})$.
By  \cref{prop:F_converge}, we obtain a $1$-Lipschitz map $F:X\to Y$.
Again the desired conclusion follows from Theorem \ref{thm:LV-rigid}.

In the general case for $\ind_\pi(f)>1$ we use
covering theory
to lift $f$ to a Lipschitz map $\hat{f}:X\to \hat{Y}$ from $X$ to the finite cover $\hat{Y}$ corresponding to $f_*\pi_1(X)<\pi_1(Y)$, with its induced metric locally isometric to that of $Y$.
In particular, $\ind_\pi(\hat{f})=1$, and we have $\Vol(\hat{Y})=\mf{m}(X)$ so we may apply the index one case to obtain an isometry $\hat{F}$. Equivariance under the deck group implies that $\hat{F}$ descends to a Riemannian cover $F:X\to Y$.
\end{proof}

\section{Applications}\label{sec:stability_proof}

In this section we present the proof of \cref{thm:stab} and provide an application to Einstein $4$-orbifolds.

\subsection{Proof of Theorem \ref{thm:stab}}
Before providing a proof of Theorem \ref{thm:stab}, we state some remarks.

\begin{Remark}\label{rmrk-differencehyp}

Recall that
Theorem \ref{thm:stab}
can be seen as an extension of
\cite[Theorem 1.3]{Bess-BCG}.
Though, in \cite[Theorem 1.3]{Bess-BCG}, the diameter bound hypothesis is on the target space $Y$, while ours is on the domain $X$.
We use our diameter assumption to obtain a uniform lower bound on the volume of $X$.
In the case of manifolds one can obtain such a bound via Gromov's Isolation Theorem \cite[pg. 14]{Gromov82}.
Currently there is no such result for $\RCD$-spaces. In the case that the spaces $X$ are in addition orientable Alexandrov spaces, there is such a result by Mitsuishi--Yamaguchi \cite[Theorem 1.8]{Mitsuishi-Yamaguchi:14}.
Using the degree properties of the simplicial volume on $X$ a uniform lower bound on the volume of $X$ can be obtained, and therefore the diameter bound on $X$ can be replaced with one on $Y$.
\end{Remark}

\begin{Remark}\label{rmrk-Examples}
Regarding the importance of
 \cref{thm:stab},
note that there are numerous examples of non-manifold $\RCD$ spaces---that are not Alexandrov spaces---which are arbitrarily close in the measured Gromov--Hausdorff sense to a manifold $M$ and are not even homotopic to $M$ (see e.g. \cite{Colding:20}).
Also, note that even in the case when $X$ and $Y$ are smooth manifolds, we cannot conclude from  \cref{thm:stab} that their smooth structures are the same as the hyperbolic smooth structures.
Although in that case, provided that $\epsv_0$ is small enough, it does follow that the smooth structures are the same (cf. \cite{SchroederEtAl:18}).
\end{Remark}

\begin{proof}[Proof of Theorem \ref{thm:stab}]
Suppose that for some $N\geq 3$, $K \in \R$, and $D>0$, no such $\eps_0(N,K,D)$
as in the statement exists.
Then, there exists a sequence $\{\epsv_i\}$  of positive numbers converging to $0$ such that for each $i$ there exist,

\begin{itemize}
\item[(i)] A compact locally $\CAT(-1)$ and $\RCD(K,N)$ space $(Y_i,d_i,\mc{H}^{N})$ without boundary;
\item[(ii)] A compact $\RCD(-(N-1),N)$ space $(X_i,d_i,\mf{m}_i)$ without boundary and with $\diam(X_i)\leq D$;
\item[(iii)] A continuous map $f_i:X_i\to Y_i$ with $1\leq\ind_\pi(f_i)$;
\end{itemize}
such that, for all $i$,
\[
\mc{H}^{N}(X_i)\leq \ind_\pi(f_i)(1+\epsv_i)\mc{H}^{N}(Y_i),
\]
and for which some part of the subsequent conclusion fails. As in the beginning of Subsection \ref{sec:natural},
we may assume as before that each $f_i$ is Lipschitz.

As $Y_i$ are assumed to be locally $\CAT(-1)$ spaces, by Lemma \ref{lem:target_space} they are also Alexandrov spaces and smooth topological manifolds. (Specifically they are $\CBB(\beta)$ with $\beta \geq K+(N-2)$.) Moreover, for any $\eps_i>0$ the metric on $Y_i$ is $(1+\eps_i)$-Lipschitz close to a Riemannian metric $g_i$ of bounded curvature between $[-1+\eps_i,K+(N-2)-\eps_i]$. By the Margulis--Heintze theorem (e.g. see \cite{Hei76})  the volumes of the $(Y_i, g_i)$ satisfy $\vol_{g_i}(Y_i)\ge C(N,K)>0$.
Therefore, we have
\begin{equation}\label{eqn:HnY-lower-bd}
\mc{H}^N(Y_i)\geq (1+\eps_i)^N\vol_{g_i}(Y_i)\ge (1+\eps_i)^NC.
\end{equation}
Analogously to how we defined $\bar{X}$, we set $\bar{X}_i=\til{X}_i/\ker (f_{i})_*$. Recall that, by Sections \ref{sec:natural} and \ref{subsec:Fs-Lip}, given $s_i>h(\bar{X}_i)$ sufficiently close to $h(\bar{X}_i)$, the natural maps $F_i:=F_{s_i}:X_i\to Y_i$ are $(1+\eta_i)$-Lipschitz maps which are homotopic to $f_i$.
Therefore, the diameter bound on $X_i$ allows us to bound the diameter of $Y_i$ by $D(1+\eta_i)$.

Let $(Y, d_Y, \mf m_Y)$ be a mGH limit of the $(Y_i,d_i,\mc{H}^{N})$, then it is a non-collapsed $\RCD(K,N)$ and $\CAT(-1)$ space.
There is a homeomorphism $\delta:[0,\infty)\to[0,\infty)$, depending only on the limit $Y$, such that there are $(1+\delta(\eps_i))$-biLipschitz homeomorphisms $\phi_i:Y_i\to Y$ for all sufficiently large indices $i$.

Now, by Theorem~\ref{thm:schwarz}, $\mc{H}^N(X_i)\geq C >0$ uniformly for all $i$.
Hence $(X_i,d_i,\mf{m}_i)$  are uniformly non-collapsed $\RCD(-(N-1),N)$
with bounded diameters.

By \cref{eq:fullIneq}
we have
\[
\frac{h(\bar{X}_i)}{h(\til{Y}_i)}\geq \ind_\pi(f_i)\frac{\mc{H}^N(Y_i)}{\mc{H}^N(X_i)}.
\]
By our assumption, this latter quantity is larger than $1+\eps_i$, which tends to $1$. Since by hypothesis $h(\bar{X}_i)\leq N-1$ and $h(\til{Y}_i)\geq N-1$, we have $h(\bar{X}_i)$ and $h(\til{Y}_i)$ must both tend to $N-1$ as $\eps_i\to 0$. Hence $s_i$ tends to $N-1$ and $\eta_i$ tends to $0$ as $i\to \infty$.

 If we choose a sequence $\set{s_{i}}$ converging to $N-1$ from above, then after passing to a subsequence the maps $F_i:X_i\to Y_i$ converge to a $1$-Lipschitz map $F:X\to Y$ by the generalized Arzela--Ascoli Theorem \cite[Appendix]{GroPet}.

Next, we show:

\begin{lemma}\label{lem:ind-pi-bounded}
We have  $\limsup_i \ind_\pi(f_i)<\infty$.
\end{lemma}

\begin{proof}
The Bishop--Gromov theorem gives the following upper bound on the volume of $X_i$:
\[
\mc{H}^N(X_i)\leq V^+=\inf_{x\in X_i,0<\eps}\frac{\mc{H}^{N}(B(x,\eps))}{\mc{H}^{N}(B_{K}(\eps))}\mc{H}^{N}(B_{K}(D))
\]
Since $X_i$ are non-collapsed $\RCD(K,N)$ spaces, by  \cite[Corollary 2.13]{DePhilippisGigli},
\[
V^+\leq \mc{H}^{N}(B_{K}(D)).
\]

The desired estimate now follows from \cref{eq:Thm2-ineq}
and \cref{eqn:HnY-lower-bd} above. \end{proof}

We continue with:
\begin{lemma}\label{lem:ind-pi-lim}
After passing to a further subsequence we have
\[
\ind_\pi(F)=\lim_i \ind_\pi(f_i).
\]
\end{lemma}

\begin{proof}
Set $k=\liminf_i \ind_\pi(f_i)$, and by \cref{lem:ind-pi-bounded} we have
\[
1\leq k \leq\limsup_i \ind_\pi(f_i)<\infty.
\]
Passing to a further subsequence, we may assume that $k=\ind_\pi(f_i)$ is constant. Since the index of $(f_i)_*\pi_1(X_i)$ is $k$, after passing to a further subsequence we may assume that $(\phi_i\of F_i)_*\pi_1(X_i)=(\phi_i\of f_i)_*\pi_1(X_i)$
is a common subgroup $\Gamma_0<\pi_1(Y)$ for all $i$. (Recall here $\phi_i:Y_i\to Y$ are the biLipschitz  homeomorphisms defined above.)

We claim that for the limit map $F:X\to Y$, $F_*\pi_1(X)>\Gamma_0$
and thus $0<\ind_\pi(F)\leq k$.  Fix $\sigma\in \Gamma_0$  with $\sigma\neq 1$. We observe that any representative of $\sigma$ has at least some length $L$ in $Y$.
Since $\phi_i\of F_i$ is $(1+\delta(\eps_i))(1+\eta(\eps_i))$-Lipschitz, for some function with $\lim_{\eps\to 0}\eta(\eps)=0$, any representative curve $c_i$ of $\gamma_i\in (\phi_i\of F_i)_*^{-1}(\sigma)\subset\pi_1(X_i)$ will have a uniformly large lower bound $L'$ for its length, independent of $i$.

We follow the notation from \cite{Mondino-Wei} and set $G(X_i,\delta):=\pi_1(X_i,p_i)/\pi_1(X_i,\delta,p_i)$.
Choosing $i$ sufficiently large so that $\eps_i<\frac{\delta}{30}$ we have that $G(X_i,\delta)$ is naturally isomorphic to $G(X,\delta)=\pi_1(X,p)/\pi_1(X,\delta,p)$.
By Theorem 2.7 of \cite{Mondino-Wei} and Theorem 3.1 of \cite{Mondino-Wei} we obtain that there is a $\delta_0(X)>0$ such that for all $\delta<\delta_0(X)$, $\pi_1(X)\cong G(X,\delta)$.
Hence, passing to a subsequence of $\gamma_i$ we may find a  subsequence of the representatives $c_i$, converging to a representative $c$ (also of length at least $L'$), of a nontrivial element $\gamma\in G(X,\delta)\cong\pi_1(X)$.
Passing to another subsequence we obtain that the images $F_i(c_i)$ converge to a representative of $\alpha$.
Hence we also have $\gamma\in\pi_1(X)$ for which $[F(c)]=F_*(\gamma)=\alpha$. The claim follows.

By the general result \cite[Corollary 1.2]{Wang2022} under the assumption of our diameter bounds on $X_i$, there is a surjective map $r_i:\pi_1(X_i)\to \pi_1(X)$.
Moreover, $r_i(\alpha)$ is realized by taking the equivalence class of a curve that is a nearby curve in $X$ to a realization of a curve in $X_i$ representing $\alpha$.
It follows that $F_*\pi_1(X)< (F_i)_*(\pi_1(X_i))=\Gamma_0$.
Hence $F_*\pi_1(X)=\Gamma_0$, and thus $\ind_\pi(F)=k$.
\end{proof}

\begin{Remark}
By hypothesis and Theorem \ref{thm:entrig}, in our case it follows that
\[1\leq \frac{\mc{H}^{N}(X_i)}{k \mc{H}^N(Y)}\leq 1+\eps_i\]
where $k=\lim_i \ind_\pi(f_i)$.
Hence $\mc{H}^{N}(X)=\lim_i\mc{H}^{N}(X_i)=k\mc{H}^N(Y)$.

Note for this last equality, we require lower curvature bounds and noncollapsing in our {\em measured}-Gromov--Hausdorff convergence, since it is not true that $\mc{H}^{N}$ is lower-semicontinuous under GH convergence, even for manifolds.
In fact, Ivanov provides an example  of metrics $g_i$ on $S^3$ that Gromov--Hausdorff converge to the round metric, but with $\Vol(S^3,g_i)\to 0$ \cite{Iva}.
(However, these do not have uniform lower curvature bounds.)
\end{Remark}

To finish the proof of \cref{thm:stab} we observe that, by Theorem \ref{thm:schwarz}:
\[
\ind_\pi(F)\mc{H}^N(Y)\leq \mc{H}^N(X)
\]
Therefore,
\[
\ind_\pi(F)\mc{H}^N(Y)\leq \mc{H}^N(X)\leq \ind_\pi(f_i)(1+\epsv_i)\mc{H}^N(Y).
\]
Taking the liminf in $i$ on both sides of the inequality and using that, for sufficiently large $i$, $\ind_\pi(f_i)=\ind_\pi(F)$ we obtain $\mc{H}^N(X)=\ind_\pi(F)\mc{H}^N(Y)$. Then, by the rigidity case in Theorem \ref{thm:schwarz}, $\ind_\pi(F)\in \mathbf{N}$ and $f$ is homotopic to a degree $\ind_\pi(F)$ Riemannian cover $X=\hat{Y}\to Y$, and both the metrics of $X$ and $Y$ are locally hyperbolic. Moreover,
\[
\mc{H}^N(\hat{Y})=\ind_\pi(F)\mc{H}^N(Y)
\]
and therefore,
\[
\mc{H}^N(X)=\mc{H}^N(\hat{Y}).
\]
However, since we now know that the $X_i$ converge to a smooth Riemannian manifold, Theorem 1.1 of \cite{Honda-Peng:22} (generalizing the main result of \cite{Colding:20}) implies that the $X_i$ are eventually homeomorphic to $X$.
This contradicts our assumption that the conclusion of the theorem fails. Hence an $\eps_0(N,K,D)>0$ with the stated properties exists.
\end{proof}

\subsection{Application to Einstein $4$-Orbifolds}
We say that $X$ is a {\em Einstein orbifold} if there is a cover of $X$ by open sets $\set{U_\alpha}$ such that $U_\alpha$ may be isometrically identified as $U_\alpha=V_\alpha/\Gamma_\alpha$ where  $V_\alpha$ is open in a common Einstein manifold $M$ and $\Gamma_\alpha$ is a discrete group of isometries, possibly with torsion.

Recall that the {\em orbifold Euler characteristic $\chi_{orb}(X)$} of an orbifold $X$ is the orbifold-equivariant homotopy invariant, defined by Satake \cite{Satake} as:
\[
\chi_{orb}(X)=\sum_{k}\frac{(-1)^{\dim s_k}}{N_{s_k}}\in\Q
\]
Here $\bigcup_k s_k$ is an equivariant triangulation of $X=\bigcup_k \abs{s_k}$, i.e. all of the irreducible components of singular points occur as subcomplexes, and $N_{s_k}$ is the order of the stabilizer of the simplex $s_k$.
We continue to denote the standard Euler characteristic of $X$ by $\chi(X)$.

\begin{Corollary}\label{cor:Einstein}
Suppose a closed $4$-dimensional Einstein orbifold $X$ with negative Einstein constant admits a continuous map $f:X\to Y$ into a hyperbolic $4$-manifold $Y$ with $\chi_{orb}(X)\leq \ind_\pi(f) \chi(Y)$.
Then $X$ is homothetic to a degree $\ind_\pi(f)$ cover of $Y$.
\end{Corollary}

\begin{proof}
In dimension four we may use the decomposition of the Pfaffian of the curvature tensor into components involving the Weyl tensor $W_g$ and the scalar and Ricci curvature components.
The corresponding decomposition in the Gauss--Bonnet--Chern formula \cite[Theorem 2]{Satake}  for the orbifold characteristic is the following,
\[
\chi_{orb}(X)=\frac{1}{8 \pi^{2}} \int_{X}\left(\left\|W_{g}\right\|^{2}-C'\left\|\op{Ricci}(g)-\frac{\op{scal}(g)}{n}g\right\|^{2}+C\left|\op{scal}{(g)}\right|^{2}\right) d \vol_{g}
\]
for some universal constants $C>0$ and $C'>0$.
The middle term is $0$ because $X$ is Einstein.
We may scale $g$ so that $\op{Ricci}(g)=-3g$.
The Weyl tensor vanishes for the constant curvature $-1$ metric $g_0$ on $Y$, and so we may estimate,
\begin{align*}
		\chi_{orb}(X) & \geq \frac{C}{8 \pi^{2}} \int_{X}\left|\op{scal}{(g)}\right|^{2} d \vol_{g}=\frac{18C}{ \pi^{2}}\operatorname{Vol}(X, g) \geq \ind_H(f)\frac{18C}{ \pi^{2}} \operatorname{Vol}\left(Y, g_{0}\right)\\
		&=\ind_H(f)\frac{1}{8 \pi^{2}} \int_{Y}C\left|\op{scal}{(g_0)}\right|^{2}d \vol_{g}=\ind_H(f)\chi(Y)\geq\ind_\pi(f)\chi(Y).
\end{align*}
Here the middle inequality follows from Theorem \ref{thm:schwarz}.

As $\chi_{orb}(X)\leq \ind_\pi(f) \chi(Y)$, there is equality and thus $\Vol_g(X)=\Vol_{g_0}(Y)$.
Hence by Theorem \ref{thm:schwarz}, $(X,g)$ is a Riemannian cover of $(Y,g_0)$.
\end{proof}

\begin{Remark}
In the definition of $\ind_\pi$ we use the standard fundamental group and not the orbifold fundamental group $\pi_1^{orb}$. We are not certain if the statement holds if we replace $\ind_\pi(f)$ by $\ind_{\pi_1^{orb}}(f)$.

Observe that it follows from the proof of the above corollary that any $X$ satisfying the hypotheses must have $\chi_{orb}(X)>0$, and thus any map satisfying the hypotheses must have $\ind_\pi(f)>0$.
\end{Remark}

\section{Proof of Proposition \ref{prop:Jac_est}}\label{sec:proof4.8}
For convenience we restate the proposition here.

{
\renewcommand{\theTheorem}{\ref{prop:Jac_est}}
\begin{Proposition}
For all $s>h(\bar{X})$, the natural map $F_s:X\to Y$
has the following inequality in the case $Y$ is negatively curved locally symmetric:
\[
\Jac F_s(x)\leq \left(\frac{s}{h(\til{Y})}\right)^{N},
\]
and
\[
\Jac F_s(x)\leq \left(\frac{s}{N-1}\right)^{N},
\]
in the case that $Y$ is as in \cref{thm:schwarz}.
\end{Proposition}
\addtocounter{Theorem}{-1}
}

\cref{prop:Jac_est} is a generalization, respectively, of \cite[Lemme 7.2, Lemme 7.4]{Besson-Courtois-Gallot:95} and Theorem 1.2 item (i) of \cite{BCG-acta} where they appear with a different normalization on the metric.
In the earlier reference, Besson--Courtois--Gallot use calibration techniques to obtain the Jacobian bounds.
Since that time, their proof has been distilled to some degree by various authors
and our approach is a variation of the later techniques which we adapt for the $\RCD$ setting.

First we begin with the next lemma.

\begin{lemma}
The map $\til{F}_s: \bar{X}\to \til{Y}$ given by $\til{F_s}(x)=\bary({\sigma}_{x}^{s})$,  where ${\sigma}_{x}^{s}=\til{f}_*\mu_x^s$, is differentiable a.e.
Furthermore, its differential can be written as
\[
d_x\til{F_s}=s (L_x^s+K_x^s)^{-1}\of A_x^s,
\]
where $A_x^s$, $K_x^s$, and $L_x^s$ are defined by \cref{A_x^s}, \cref{K_x^s}, and \cref{L_x^s} respectively.
\end{lemma}

\begin{proof}
 We have already established the locally Lipschitz property by Lemma \ref{lem:Lips}.

As $\til{Y}$ is a Hadamard space and since $\mf{m}$, and hence ${\sigma}_{x}^{s}$, are
nonatomic, for any $x\in \bar{X}$ the function $\mc{B}_{s,x}:=
\mc{B}_{\sigma_x^{s}}:\til{Y}\to \R$ defined as in (\ref{Busemann like function}) is smooth on
$\til{Y}$ with gradient,
\[
\nabla_y{\mc{B}}_{s,x}=\int_{\til{Y}}\rho_z\nabla_y \rho_z\, d{\sigma}_{x}^{s}(z).
\]
Here $\rho_z$ is the function $\rho_z(y)=d(y, z)$ on ${\til{Y}}$ which is globally 1-Lipschitz and weakly differentiable.
In particular, for $x\in \bar{X}$, we
have the defining equation
\[
\nabla_{{\til{F}}_{s}(x)}{\mc{B}}_{s,x}=0.
\]

Assume we are at a point $x\in \bar{X}$ where $T_x\bar{X}=T_xX$ is defined. Then by Rademacher theorem  $\nabla_{x'} \, d(x,x')$ exists for a.e. $x'\in \bar{X}$.
Let
\[
d_x d{\sigma}_{x}^{s}(z)=-s G_{x,z}d{\sigma}_x^s(z),
\]
 for a $1$-tensor $G_{x,z}$.

 That is, we can define
\[
-s G_{x,z}=d_x\left(\frac{d{\sigma}_{x}^{s}}{d{\sigma}_{p}^{s}}(z)\right) \frac{d{\sigma}_{p}^{s}}{d{\sigma}_{x}^{s}}(z),
\]
for any fixed choice of $p\in \bar{X}$.
Since ${\sigma}_{x}^{s}=\til{f}_*(e^{-s d(x,\,\cdot )}{\tmX}),$
we have for $u\in T_x\bar{X}$,

\begin{equation}\label{eqn:G-formula}
G_{x,z}(u)=\lim_{\eps\to 0}\frac{ \int_{\til{f}^{-1}(B(z,\eps))}  \del_u d(x,x') e^{-s d(x,x')}d{\tmX}(x')}{\int_{\til{f}^{-1}(B(z,\eps))} e^{-sd(x,x')}d{\tmX}(x')}.
\end{equation}
Here we have understood the gradient of $d$ in the weak sense (see \cite{AGS-08, GigliMondinoRajala}).
Moreover, $d$ is 1-Lipschitz, and since $G_{x,z}$ is an average of weak 1-tensors of unit norm, we have $\norm{G_{x,z}}\leq 1$.

Consider the map $x\mapsto \nabla_{{\til{F}}_{s}(x)}{\mc{B}}_{s,x}$ as a map from $\bar{X}$ into vector fields on $\til{Y}$, which in this case happens to vanish. Differentiating $\nabla_{{\til{F}}_{s}(x)}{\mc{B}}_{s,x}$ with respect to $x$ (that is, with respect to the connection on $Y$ and the generalized differential structure on $X$) in the direction $u\in T_x\bar{X}$ yields for $\mf{m}$-a.e. $x$ and $u\in T_xX$,
\begin{align*}
0&=d_x \nabla_{\til{F_s}(x)} {\mc{B}}_{s,x}(u)\\
&=\int_{\til{Y}} (\nabla_{y}\rho_z\tensor \nabla_{y}\rho_z + \rho_z(y)D\nabla_{y}\rho_z)|_{y=\til{F_s}(x)}  \of
d_x\til{F_s}(u)d{\sigma}_{x}^{s}(z)\\
&\hspace{2cm} + \int_{\til{Y}}\rho_z(\til{F_s}(x))\nabla_{\til{F_s}(x)} \rho_z
\tensor d_x \left( \frac
{d{\sigma}_{x}^{s}}{d{\sigma}_{p}^{s}}(z)\right)(u)\,
d{\sigma}_{p}^{s}(z)\\
&= \left(\int_{\til{Y}}
(\nabla_{y}\rho_z\tensor \nabla_{y}\rho_z +\rho_z(y) D\nabla_{y}\rho_z)|_{y=\til{F_s}(x)}d{\sigma}_{x}^{s}(z)\right)\of
d_x\til{F_s}(u)\\
&\hspace{2cm} - s\int_{\til{Y}}\rho_z(\til{F_s}(x))\nabla_{\til{F_s}(x)} \rho_z \tensor G_{x,z}(u)\,
d{\sigma}_{x}^{s}(z)\\
&=\norm{\hat{\eta}_x^{\, s}}\left(\int_{\til{Y}}
\left( \frac{1}{\rho_z(y)}\nabla_{y}\rho_z\tensor \nabla_{y}\rho_z +D\nabla_{y}\rho_z\right) |_{y=\til{F_s}(x)}d{\eta}_{x}^{s}(z)\right)\of
d_x\til{F_s}(u)\\
&\hspace{2cm} - s\norm{\hat{\eta}_x^{\, s}}\int_{\til{Y}}\nabla_{\til{F_s}(x)} \rho_z \tensor G_{x,z}(u)\,
d{\eta}_{x}^{s}(z),
\end{align*}
where $\hat{\eta}_x^{\, s}$ and $\eta_x^s$ are the measures defined as
\begin{equation}
d\hat{\eta}_x^{\, s}(z)=\rho_z(\til{F}_s(x))d\sigma_x^s(z), \ \ \ \ \eta_x^s=\frac{\hat{\eta}_x^{\, s}}{\norm{\hat{\eta}_x^{\, s}}}. \label{eta-measures}
\end{equation}
Observe that $\eta_x^s$, and the integrals above exist provided that $\sigma_x^s$ is not an atom at a single point, say $g(x)$.
(In that case, we would have $\til{F}_s(x)=g(x)$.)
However, by construction $\sigma_x^s$ is never of this form.

We note in the above formula we are using the fact that by item (4) of \cref{lem:target_space}, $\rho_z\in W^{2,1}(\til{Y})$ and in particular the term $D\nabla_{y}\rho_z$ is integrable.

Here, as before, the tensor in the last term is defined for almost every $x$ where the weak differential structure exists, since $TY\tensor T^*X$ makes sense there.
(We will not need to concern ourselves with lower dimensional strata where $d_x \nabla_{\til{F_s}(x)} {\mc{B}}_{s,x}$ has nontrivial kernel, because these have been shown to have measure $0$ by Bru\`e--Semola \cite{brue-semola:18}.)

Notice that all the associated objects exist (at least weakly in $L^1$).
The distance function on $X$ is weakly differentiable and $X$ has a tangent space at $x$, so $G_{x,z}(u)$ will exist for $\mf{m}$-a.e. $x\in X$ and every $z\in Y$, and moreover it has at most unit norm when defined.
Hence this makes sense under the integral.

At a.e. point $x\in X$, where the appropriate derivatives above exist, we define the operators $A_x^s:T_{x}X\to T_{\til{F_s}(x)}\til{Y}$  and $L_x^s,K_x^s:T_{\til{F_s}(x)}\til{Y}\to T_{\til{F_s}(x)}\til{Y}$ by,
\begin{equation}\label{A_x^s}
  A_x^s(u):=\int_{\til{Y}}\nabla_{\til{F_s}(x)} \rho_z \tensor G_{x,z}(u)\,
d{\eta}_{x}^{s}(z),
\end{equation}

\begin{equation}\label{K_x^s}
K_x^s(v):=\int_{\til{Y}}
\left(  D_v\nabla_{y}\rho_z\right) |_{y=\til{F_s}(x)}d{\eta}_{x}^{s}(z), \quad\text{and}\quad
\end{equation}
\begin{equation}\label{L_x^s}
L_x^s(u):=\int_{\til{Y}}
\left( \frac{1}{\rho_z(y)}\nabla_{y}\rho_z\tensor d_{y}\rho_z(v)\right)|_{y=\til{F_s}(x)}d{\eta}_{x}^{s}(z).
\end{equation}

We can use that $0=D_u \nabla_{\til{F_s}(x)} {\mc{B}}_{s,x}$ to formally solve for $d_x\til{F_s}(u)$, which yields
\begin{equation}
(L_x^s+K_x^s)\of d_x\til{F_s}-s A_x^s=0,
\end{equation}
or
\begin{equation}\label{eq:dF=sKA}
d_x\til{F_s}=s (L_x^s+K_x^s)^{-1}\of A_x^s\leq s (K_x^s)^{-1}\of A_x^s,
\end{equation}
where the last inequality should be interpreted as for two-forms when evaluated on pairs of vectors, and this holds since $L_x^s$ is positive semi-definite.

In fact, whenever $T_xX$ exists and $K_x^s$, $L_x^s$ and $A_x^s$ are differentiable, which simultaneously holds for $\mf{m}$-a.e. $x\in X$, the chain rule for Lipschitz maps (e.g. Theorem 2.1 of \cite{AmbrosioDalMaso90}) implies that $\til{F_s}$ is differentiable at $x\in X$ as well and \cref{eq:dF=sKA} gives $d_x\til{F_s}$.
\end{proof}

Now we can continue with the proof of \cref{prop:Jac_est}:

\begin{proof}
Let $\op{II}_y$ denote, when defined, the second
fundamental form (at the point $y\in \til Y$) of the sphere of radius  $\rho_z(y)$ centered at
$z$, operating on its tangent space. The $(1,1)$-form $D\nabla_y \rho_z$ is just the second fundamental form extended to equal $0$ in the normal $\nabla_y \rho_z$ direction, i.e. $D\nabla_y \rho_z=\op{II}_y\oplus 0$. Observe that when defined the form $D\nabla_y \rho_z$ is positive
definite except in the null direction $\nabla_y \rho_z$, because the spheres in any Hadamard Alexandrov space are the boundaries of strictly convex balls. Recall that $\til{Y}$ is a smooth manifold and so also by the convexity of $\rho_z$, $D\nabla_y \rho_z$ is defined at a.e. $y$ even though the metric is only $C^{1,\alpha}$.

The measure ${\sigma}_{x}^{s}$ is non atomic, and not concentrated on any single geodesic. Hence the $\eta^s_x$ average over $y$ of the positive-semidefinite forms, $D\nabla_y \rho_z$, will be strictly positive definite. However, this average will not necessarily be uniformly bounded away from 0 independent of the measure $\eta_x^s$ or the geometry of $\til{Y}$.

Following the technique introduced by Besson, Courtois and Gallot, we will show that the product of sufficiently many of the small singular values of the $A^s_x$ tensor control the single---potentially small---eigenvalue of $K_x^s$.

To understand the integrand of $K_x$, we observe from the constant
curvature $-k^2\leq -1$ case, that the solutions to the Riccati equation imply that the
second fundamental form at any point $y$ of a sphere of radius $t$ is
\[
\op{II}_y=k\coth(k t)I\geq \coth(t)I.
\]
By item (4) of Lemma~\ref{lem:target_space}, we have that $\op{II}_y \geq \coth(d(y,z))I$ on the $\CAT(-1)$ space $\til{Y}$.
Hence $K_x^s$ has full rank at each $x$ where it is defined.

Set $y=\til{F}_s(x)$, and $G_{x,z}^*$ to be the weak cotangent 1-form to $G_{x,z}$ (see \cite{Gigli-diff}).
By Cauchy--Schwarz applied to bilinear forms, we may write,
\begin{equation}\label{eq:AHB}
(A_x^s)^*A^s_x\leq H_x^s \of B_x^s,
\end{equation}
where
\begin{align}\label{eq:tensors}
H_x^s=\int_{\til{Y}}\nabla_y \rho_z \tensor d_y\rho_z \, d{\eta}_{x}^{s}(z), \quad\text{and}\quad
B_x^s=\int_{\til{Y}}G_{x,z}^*\tensor G_{x,z}\, d{\eta}_{x}^{s}(z).
\end{align}

Hence $d_xF_s \leq s(K_x^s)^{-1}H_x^s\of B_x^s$ as (1,1)-forms.

The determinant of $B_x^s$ can be estimated by noting that if it were smooth the trace of the integrand of $B_x^s$ would be at most one, and that among positive semi-definite symmetric matrices of trace $1$, the determinant is maximized at $\frac{1}{N}I$.
The same estimate can be made weakly for the entire integral of the weak gradients.
Observe that the trace of the integral will again be at most $1$ because $G_{x,z}$ is the derivative of a $1$ -Lipschitz function and the tensors of the unit vectors have unit trace.
Therefore $\det B_x^s\leq \frac{1}{(N)^{N}}$.
Consequently:
\begin{align}\label{eq:Jac-est1}
\begin{split}
\left(\Jac \til{F_s}(x)\right)^2&=\det(s (L_x^s+K_x^s)^{-1}\of A_x^s)^2\leq s^{2N}\frac{\det (A_x^s)^2}{\det (K_x^s)^2} \\
&\leq (s)^{2N} \frac{\det
H_x^s \det B_x^s}{(\det K_x^s)^2}\leq \left(\frac{s^{2}}{N}\right)^{N} \frac{\det
H_x^s}{(\det K_x^s)^2}.
\end{split}
\end{align}
Observe that, by Lemma \ref{lem:Jacobian}, the left hand side is defined $\mf{m}$-a.e. .

To warm up, we first estimate this in the case that $Y$ is negatively curved with maximum curvature $-1$. In this case we note that as $(1,1)$-forms we have
\[
D\grad_{y}\rho_z\geq I\coth(d(y,z))-\grad_{y}\rho_z\tensor d_{y}\rho_z\geq I-\grad_{y}\rho_z\tensor d_{y}\rho_z.
\]
Hence after integrating we obtain
\[
{K}_{x}^s = \int_{\til{Y}} D\nabla_y{\rho_z}(z) d{\eta}_{x}^s(z)\geq
\int_{\til{Y}}
I-E_{y} \ d{\eta}_{x}^s( z)\\
= I-{H}_{x}^s,
\]
where $E_{y}$ is the $(1,1)$-form $\grad_{y} {\rho_z}\tensor d_{y}{\rho_z}$.

Now we estimate $K_x^s$ in the case that $\til{Y}$ is the symmetric space $\HK$ for one of the four division algebras $\bf{K}=\R,\C,\HH,\OO$.
These arguments follow a similar approach to the one applied for higher graph manifolds \cite{ConSS19} and the original argument in \cite{Besson-Courtois-Gallot:95}.
We include them here for readers convenience, with the exception of the Octonion case $\bf{K}=\OO$, for which the original argument outlined below was shown to fail (\cite{Ruan22}). Ruan provides a correction in the same paper, and that same proof works in our setting as well since it only depends on the geometry of the target which is the same in this case.

Consider the ball $B(p,R)$ and any point $z\in S(p,R)$. Set $y=\til{F}_s(x)$, then we have,
\[
D\nabla_y{\rho_z}( z)_{|_{(\grad {\rho_z})^\perp}}=\sqrt{-R_{
z}}\coth\left( {\rho_z}(
z)\sqrt{-R_{y}}\right).
\]

We are denoting by $R_{{y}}$  the $(1,1)$-form dual to  $R(\grad
{\rho_z},\, \cdot\, ,\grad{\rho_z},\,\cdot\, )$, the curvature tensor at the point $y$ twice contracted
in the direction of $\grad_y{\rho_z}$.
Recall that the curvature tensor is parallel in $\HK$, so the Riccati equation can be solved explicitly.
This computation yields the formula above.

When the  field $\bf{K}$ has real dimension $d$, there exist $d-1$ almost-complex structures $J_i:T\HK\to T\HK$, such that $J_i^{-1}=-J_i$.
Therefore,
\[
\sqrt{-R_{y}}=I-E_{y}-\sum_{i=1}^{d-1} J_i E_{y} {{J_i}{|_{(\grad
{\rho_z})^\perp}}}.
\]
Here, $E_{y}$ is once more the $(1,1)$-form $\grad_{y} {\rho_z}\tensor d_{y}{\rho_z}$.
In the direction of $\grad {\rho_z}$, we find $D\nabla_{y}{\rho_z}(y)=0$.
Hence, on the one hand,
\[
D\nabla_{ y}{\rho_z}=\left( I-E_{ y}-\sum_{i=1}^{d-1} J_i E_{ y}
{J_i}\right) \coth\left({\rho_z}( y)\left( I-E_{ y}-\sum_{i=1}^{d-1} J_i E_{ y} {J_i}\right) \right).
\]
On the other hand, because $\coth(t)\geq 1$ for $t>0$, we have,
\[
\coth\left({\rho_z}( y)\left( I-E_{ y}-\sum_{i=1}^{d-1} J_i E_{ y} {J_i}\right) \right)\geq I,
\]
as positive definite symmetric two forms.
For any $R>0$ for which $B\left( \til{F_s}(x),R \right) \subset \til{Y}$ is isometric to $B({p},R)\subset \HK$ a comparison measure ${\sigma}_{x}^s$ can be constructed on $\HK$, defined by ${\sigma}_{x}^{s}$ on the set $B({p},R)$ and $0$ outside it.
Notice that  by definition this measure will be strictly smaller.

The action of the maximal compact subgroup $K<{\rm
Isom}(\HK)$ commutes with the  maps $J_i$. So after integrating we obtain,
\begin{align*}
{K}_{x}^s &= \int_{\HK} D\nabla_y{\rho_z}( z) d{\eta}_{x}^s( z)\geq
\int_{\HK}
I-E_{ y}-\sum_{i=1}^{d-1} J_i E_{ y}J_i \ d{\eta}_{x}^s( z)\\
&= I-{H}_{x}^s-\sum_{i=1}^{d-1} J_i {H}_{x}^s J_i.
\end{align*}
Remember the previous definition used here:
\[
{H}_{x}^s:= \int_{\HK} E_{ y} d{\eta}_{x}^s( z)
\]

Substitution of this lower bound for $K_x^s$ into \cref{eq:Jac-est1} gives,
\begin{equation}
  \left(\Jac \til{F_s}(x)\right)^2\leq \left(\frac{s^{2}}{N}\right)^{N} \frac{\det
H_x^s}{\det(I-{H}_{x}^s-\sum_{i=1}^{d-1} J_i {H}_{x}^s J_i)^2}. \label{Jac-estimate-1}
\end{equation}
The 2-form ${H}_x^s$ is also strictly positive definite, because the measure ${\eta}_{x}^{s}$
is nonatomic.
The next lemma then completes our proof of \cref{prop:Jac_est} (observe that it can also be applied to the non-symmetric case, using $d=1$).
\end{proof}

\begin{Lemma}[Proposition B.1 and B.5
of \cite{Besson-Courtois-Gallot:95}]\label{lem:BCG_est}
For all $N\times N$ ($N\ge 3$) positive definite matrices $H$ with trace one, and orthogonal
matrices $J_1,\dots, J_{d-1}$ with $J_i^2=-I$ we have
\begin{align}
\frac{\det H}{\det(I-H -\sum_{i=1}^{d-1} J_i H J_i)^2} & \leq
\left(\frac{N}{(N+d-2)^{2}}\right)^N \left(1-A \sum_{j=1}^{N}\left(\mu_{j}-\frac{1}{N}\right)^{2}\right)^2 \label{alg-ineg1} \\ & \le
\left(\frac{N}{(N+d-2)^{2}}\right)^N  \label{alg-ineg2}
\end{align}
for some positive uniform constant $A>0$.
Here $0<\mu_j<1$ are eigenvalues of $H$.
Equality in \cref{alg-ineg2} occurs if and only if $H=\frac{1}{N}I$.
\end{Lemma}
Note that the entropy of ${\bf H_K}$ is equal to $N+d-2$, while the sectional curvature is pinched between $-4$ and $-1$.

\section{Proof of Proposition \ref{prop:F_converge}}\label{sec:FinW1infty}

The aim of this section is to give the proof of Proposition \ref{prop:F_converge}. For convenience we restate it here.

{
\renewcommand{\theTheorem}{\ref{prop:F_converge}}
\begin{Proposition}
In the equality case of \cref{thm:entrig} and \cref{thm:schwarz} and when $\ind_\pi(f)=1$, there exists a decreasing sequence $s_i\to h(\bar{X})$ such that  $F_{s_i}$ converges to a 1-Lipschitz map $F:X\to Y$ homotopic to $f$.
\end{Proposition}
\addtocounter{Theorem}{-1}
}

The proof of \cref{prop:F_converge} depends on several key steps.
First, relying on the fact that the maps $F_s$ are Lipschitz (\cref{lem:Lips}), we show that the bounds established in \cref{sec:proof4.8} together with additional estimates analogous to those in Appendix A of \cite{Sambusetti-1999} which generalize \cite[Lemma 7.5]{Besson-Courtois-Gallot:95}, and \cref{lem:C-Lip} give us uniform Lipschitz control independent of $s$.

In what follows we denote by $h_0=N-1$ in the equality case of \cref{thm:schwarz} or $h_0=N+d-2$ in the equality case of \cref{thm:entrig}. We will assume from now on that $\ind_\pi(f)=1$.

We first note that from \cref{prop:Jac_est}, the coarea formula, and the equality assumption that ${h(\bar{X})}^N\mc{H}^N(X)={h_0}^N\mc{H}^N(Y)$, we have for any sequence $s_i\searrow h(\bar{X})$,
\begin{align*}
\left(\frac{s_i}{h_0}\right)^N\mc{H}^N(X) &\geq\int_{X} \Jac_x F_{s_i} d\mathcal{H}^{N}(x) = \int_{Y} \#(F_{s_i}^{-1}(y))\ d \mathcal{H}^{N}(y) \\
&\geq\mc{H}^N(Y) = \left(\frac{h(\bar{X})}{h_0}\right)^N\mc{H}^N(X).
\end{align*}
In particular, the pointwise bound $\Jac_x F_{s_i}\leq \left(\frac{s_i}{h_0}\right)^N$ from \cref{prop:Jac_est} implies that there is a sequence $\eps_i\to 0$ such that $$(1-\eps_i)\left(\frac{h(\bar{X})}{h_0}\right)^N\leq \Jac_x F_{s_i}\leq (1+\eps_i)\left(\frac{h(\bar{X})}{h_0}\right)^N$$ for a.e. $x\in X$, off of a set of $\mathcal{H}^{N}$-measure $\eps_i$.
Similarly $\#(F_{s_i}^{-1}(y))=1$ for a.e. $y\in Y$, off of a set of measure $\eps_i$.
After passing to a subsequence, we have
\begin{align}  \label{Jac-limit}
\lim_{i\to \infty}  \Jac_x F_{s_i}=\left(\frac{h(\bar{X})}{h_0}\right)^N \qquad\text{and}\qquad \lim_{i\to \infty} \#(F_{s_i}^{-1}(y))=1
\end{align}
for a.e. $x\in X$ and a.e. $y\in Y$.

Throughout we will use some of the notation introduced in \cref{sec:proof4.8} for  $A_x^s$ \cref{A_x^s}, $K_x^s$ \cref{K_x^s}, $L_x^s$ \cref{L_x^s}, $G_{x,s}$ \cref{eqn:G-formula}, and $H_x^s$ \cref{eq:tensors}.

\begin{lemma}\label{lem:H_converge}
For any sequence $s_i\to h(\bar{X})$ and for a.e. $x\in X$, the quadratic forms $H^{s_i}_x$ converge to $\frac1N I$.
\end{lemma}
\begin{proof}
We note that within the full measure set of manifold points of $X$ there is a smaller full measure subset where $F_{s_i}$ has a derivative. At these points apply \cref{alg-ineg1} to the matrix $H_x^{s_i}$ gives
\begin{equation}\label{eqn:detH-eigen-estimate}
\frac{\frac{s_i^{N}}{N^{N/2}}\left(\operatorname{det}\left(H_{x}^{s_i}\right)\right)^{1 / 2}}{\operatorname{det}\left(I-H_{x}^{s_i}-\sum_{k=1}^{d-1} J_{k} H_{x}^{s_i} J_{k}\right)} \leq\left(\frac{s_i}{h_{0}}\right)^{N}\left(1-A \sum_{j=1}^{N}\left(\mu_{j}^{s_i}(x)-\frac{1}{N}\right)^{2}\right),
\end{equation}
for some uniform constant $A>0$ and where $\mu_{j}^{s_i}$ are the eigenvalues of $H^{s_i}_x$ for $j=1,\dots,N$.
Combining \cref{eqn:detH-eigen-estimate} with the estimate in \cref{Jac-estimate-1} yields,
\[
\Jac F_{s_i}(x) \leq\left(\frac{s_i}{h_{0}}\right)^{N}\left(1-A \sum_{j=1}^{N}\left(\mu_{j}^{s_i}(x)-\frac{1}{N}\right)^{2}\right).
\]
From this we obtain,
\begin{equation}
\sum_{j=1}^{N}\left(\mu_{j}^{s_i}(x)-\frac{1}{N}\right)^{2} \leq \frac{1}{A}\left(1-\left(\frac{h_{0}}{s_i}\right)^{N}\operatorname{Jac} F_{s_i}(x)\right).
\end{equation}
By \eqref{Jac-limit} $\sum_{j=1}^{N}\left(\mu_{j}^{s_i}(x)-\frac{1}{N}\right)^{2} \underset{i \rightarrow+\infty}{\longrightarrow} 0$ almost surely.
In other words if $O_i$ diagonalizes  $H_{x}^{s_i}$, then $O_iH_{x}^{s_i}O_i^*-\frac1N I$ converges to the $0$ form and therefore  $H_x^{s_i} \underset{i \rightarrow+\infty}{\longrightarrow} \frac{1}{N} I$ for a.e. $x\in X$.
\end{proof}

In order to prove the uniform convergence of the $H_{x}^{s_i}$'s on a full measure set, we will need to study the variation of $H_{x}^{s_i}$ with respect to $x$, and to show that if $x$ and $x^{\prime}$ are enough close, then $H_{x}^{s_i}$ and $H_{ x^{\prime}}^{s_i}$ are close too.

We note that in what follows, the parallel translation in $Y$ is well defined since this depends only on the $C^1$ structure of the Riemannian metric, and indeed the metric on $Y$ is induced from a $C^{1,\alpha}$-Riemannian one by Property (3) of \cref{lem:target_space}. With this we have the following version. (We provide a proof, appropriately modified from that in \cite{Besson-Courtois-Gallot:95}, in our context for completeness.)

Let $\Omega^{s}\subset X$ be the full-measure subset where $d_xF_s$, $H_x^s$ and $K_x^s$ are well defined and the first equality in \cref{eq:dF=sKA} holds.

\begin{lemma}[Lemma 7.5b of \cite{Besson-Courtois-Gallot:95} and cf. Lemma A.6 of \cite{Sambusetti-1999}]\label{lem:Parallel_translate}
For any sequence $\set{s_i}$ converging to $h(\bar{X})$, let $x_1, x_2 \in \Omega^{s_i}$, let $q_1=\til{F}_{s_i}(x_1), q_2=\til{F}_{s_i}(x_2)$ and let $\beta$ be a minimizing $d_Y$-geodesic from $q_1$ to $q_2$. If $P_{q_2}$ denotes the parallel translation from $T_{q_1}Y$ to $T_{q_2}Y$ along $\beta$, one has:
\begin{align*}
\left\|H_{x_1}^{s_i}-P_{q_2}^{-1} \circ H_{x_2}^{s_i} \circ P_{q_2}\right\| \leq C\left[d_X(x_1, x_2)+d_Y(q_1, q_2)\right]
\end{align*}
for some constant $C$ which does not depend on $i, x_1, x_2$.
\end{lemma}
\begin{proof}
We begin by noting that since $\norm{H_{x_2}^{s_i}}\leq 1$ and $P_{q_2}$ is orthogonal, the estimate is trivial if $d(q_1,q_2)\geq 1$ so we assume $d(q_1,q_2)<1$.

Let $h_x^s(u)=\int_{\til Y}(d_{\til{F}(x)}\rho_z(u))^2 d\eta_{x}^s(z)
$ denote the $(2,0)$-form corresponding to $H_x^s$. Since $P_{q_2}^{-1}=P_{q_2}^{*}$, in dualizing the $(1,1)$-form the equivalent expression we want is,
\begin{align}
    \left\|h_{x_1}^{s_i}\right. & \left.-h_{x_2}^{s_i} \circ P_{q_2}\right\| \leq C\left[d_X(x_1, x_2)+d_Y(q_1, q_2)\right].
\end{align}

First, to estimate $\left\|h_{x_1}^{s_i}- h_{x_2}^{s_i} \circ P_{q_2}\right\|$, we let $Z$ denote a unit parallel field along the minimal geodesic  $\beta$ from $q_1$ to $q_2$.
(The existence of such a field only depends on the $C^1$ regularity of the Riemannian metric on $Y$.)
Then
\begin{align*} &
\left|\int_{\til{Y}} (d_{q_2}\rho_z(Z(q_2)))^2 d\eta_{x_2}^s(z) -\int_{\til{Y}} (d_{q_1}\rho_z(Z(q_1)))^2 d\eta_{x_1}^s(z) \right|\\
&=\left| \int_{\til{Y}} \left[ (d_{q_2}\rho_z(Z(q_2)))^2 -(d_{q_1}\rho_z(Z(q_1)))^2 \right] d\eta_{x_2}^s(z) -\int_{\til{Y}} (d_{q_1}\rho_z(Z(q_1)))^2 (d\eta_{x_1}^s(z)-d\eta_{x_2}^s(z)) \right|\\
& \leq \left| \int_{\til{Y}} \left[ (d_{q_2}\rho_z(Z(q_2)))^2 -(d_{q_1}\rho_z(Z(q_1)))^2 \right] d\eta_{x_2}^s(z)\right| + \norm{\eta_{x_1}^s-\eta_{x_2}^s}
\end{align*}
Next we show \begin{align}
    \norm{\eta_{x_2}^s-\eta_{x_1}^s}\leq C(d(x_1,x_2)+d(\til{F}_s(x_1),\til{F}_s(x_2)) \label{eq:eta-diff}
    \end{align}
    for some constant $C\geq\! 1$ depends only on $\diam X, h(\bar{X})$ for all $s \in\! (h(\bar{X}), h(\bar{X})+1]$, and thus independent of $x_1, x_2$. This estimate corresponds to \cite[(21)]{Sambusetti-1999}.

Recall the definition of $\eta$ measures in \cref{eta-measures}.
Observe
\begin{align*}
   \norm{\eta_{x_2}^s-\eta_{x_1}^s} & =  \frac{\norm{ \norm{\hat{\eta}_{x_1}^{\, s}} \hat{\eta}_{x_2}^{\, s} - \norm{\hat{\eta}_{x_2}^{\, s}}\hat{\eta}_{x_1}^{\, s}}}{\norm{\hat{\eta}_{x_1}^{\, s}} \, \norm{\hat{\eta}_{x_2}^{\, s}}}\\
   & \le \frac{\norm{\hat{\eta}_{x_2}^{\, s} -\hat{\eta}_{x_1}^{\, s}} + |\norm{\hat{\eta}_{x_1}^{\, s}} -\norm{\hat{\eta}_{x_2}^{\, s}}|}{ \norm{\hat{\eta}_{x_2}^{\, s}}}  \le 2 \frac{\norm{\hat{\eta}_{x_2}^{\, s} -\hat{\eta}_{x_1}^{\, s}}}{ \norm{\hat{\eta}_{x_2}^{\, s}}}.
\end{align*}
Also
\begin{align*}
\norm{\hat{\eta}_{x_2}^{\, s} -\hat{\eta}_{x_1}^{\, s}} & = \norm{d(\cdot,\til{F}_s(x_2))\sigma_{x_2}^s-d(\cdot,\til{F}_s(x_1))\sigma_{x_1}^s} \\
& \le \norm{d(\cdot,\til{F}_s(x_2))(\sigma_{x_2}^s -\sigma_{x_1}^s)} +  d(\til{F}_s(x_1),\til{F}_s(x_2)) \, \norm{\sigma_{x_1}^s} \\
& \le C(\diam X, h(\bar{X})) \left[\norm{\hat{\eta}_{x_2}^{\, s}} d_X(x_1, x_2) + d(\til{F}_s(x_1),\til{F}_s(x_2)) \, \norm{\sigma_{x_2}^s} \right]
\end{align*}
for $s \in (h(\bar{X}), h(\bar{X}) +1]$. Here we used the estimate that $\norm{\sigma_{x_1}^s} \le e^{s d_X(x_1, x_2)} \norm{\sigma_{x_2}^s}$ for the last term.  Hence
\begin{equation}
 \norm{\eta_{x_2}^s-\eta_{x_1}^s} \le C(\diam X, h(\bar{X})) \left[ d_X(x_1, x_2) + d(\til{F}_s(x_1),\til{F}_s(x_2)) \, \tfrac{\norm{\sigma_{x_2}^s}}{\norm{\hat{\eta}_{x_2}^{\, s}}} \right].  \label{eq:eta-diff est}
\end{equation}
Now we note that
\[
\frac{\norm{\sigma_x^s}}{\norm{\hat{\eta}_x^s}}\leq\frac{\norm{\sigma_x^s}}{\hat{\eta}_x^s(\til{Y}\setminus B(\til{F}_s(x),1))} \leq \frac{\norm{\sigma_x^s}}{\sigma_x^s(\til{Y}\setminus B(\til{F}_s(x),1))},
\]
since for $z\in \til{Y}\setminus B(\til{F}_s(x),1)$ we have $\rho_z(\til{F}_s(x))\geq 1$. Since $X= \bar{X}/\Gamma$ is compact, $\til{f}$ is a quasi-isometry on $\bar{X}$.
Hence we have $\til{f}^{-1}(B(\til{F}_s(x),1))$ belongs to a finite union of fundamental domains in $\bar{X}$, independent of the choice of $x\in X$.
This gives  that
\[
\int_{\til{f}^{-1}(B(\til{F}_s(x),1))} e^{-s d(x,z)} d\mc{H}^N(z) \leq C'\int_{\bar{X}\setminus \til{f}^{-1}(B(\til{F}_s(x),1))} e^{-s d(x,z)} d\mc{H}^N(z)
\]
for some $C'>0$ independent of $x$ and $h(\bar{X})\leq s \leq h(\bar{X})+1$.
So we obtain the following bound,
\begin{equation}\label{eq:sigma eta ratio}
\frac{\norm{\sigma_x^s}}{\norm{\hat{\eta}_x^s}}\leq 1+C'.
\end{equation}
This bound combined with \cref{eq:eta-diff est} gives \cref{eq:eta-diff}.

For the remaining term $\abs{\int_{\til{Y}} \left[ (d_{q_2}\rho_z(Z(q_2)))^2 -(d_{q_1}\rho_z(Z(q_1)))^2 \right] d\eta_{x_2}^s(z)}$, we must contend with the fact that the second fundamental form of the distance function explodes when the distance is near 0. For this reason, we need to split the analysis into two regions, where the integrating variable is in  a compact region containing  $q_1$ and $q_2$ and the remaining region.

First we analyze the compact region containing the geodesic from $q_1$ to $q_2$.

We have by the parallelism of $Z$ and the fundamental theorem of calculus that,
\begin{align}\label{eq:rho_d_rhosq}
\begin{split}
\left| \rho_z(q_2) \right. &\left.( d_{q_2}\rho_z(Z(q_2)))^2-\rho_z(q_1)(d_{q_1}\rho_z(Z(q_1)))^2 \right| \\
&\leq \left(\sup_{t} \Big| d_{\beta(t)}\rho_z(\beta'(t))(d_{\beta(t)}\rho_z(Z(\beta(t))))^2 +\right.\\
&\hspace{2cm}\left. 2\rho_z(\beta(t))\inner{\beta'(t),Dd_{\beta(t)}\rho_z(Y(\beta(t)))}\Big| \right) d_Y(q_1,q_2) \\
&\leq [1+ 2 C \rho_z(\beta(t_0))\coth( \rho_z(\beta(t_0)))] d_Y(q_1,q_2)\\
&\leq [1+ 2 C (1+\rho_z(\beta(t_0)))]d_Y(q_1,q_2)\\
&\leq [1+ 2 C (1+\rho_z(q_1)+d_Y(q_1,q_2))]d_Y(q_1,q_2),
\end{split}
\end{align}
where $C$ is the square root of the negative of the lower curvature bound, and $t_0$ is the value of $t$ achieving the supremum $\sup_{t}d(z,\beta(t))\coth(d(z,\beta(t)))$. Note by convexity of $\rho_z$, we have $\beta(t_0)\in\set{q_1,q_2}$. Here we used the Hessian comparison for the second inequality.

Let $B=B(q_1,d(q_1,q_2)+10)\subset \til{Y}$ be a fixed ball of the given radius about $q_1$. For $z\in \til{Y}\setminus B$, a similar estimate gives,
\begin{align*}
\left| \right. &\left.d_{q_2}\rho_z(Z(q_2))-d_{q_1}\rho_z(Z(q_1)) \right| \\
&\leq \left(\sup_{t} \Big| \inner{\beta'(t),Dd_{\beta(t)}\rho_z(Y(\beta(t)))}\Big| \right) d_Y(q_1,q_2) \\
&\leq [ C \sup_{t\in [0,d(q_1,q_2)]}\coth( \rho_z(\beta(t)))] d_Y(q_1,q_2)\\
&\leq 2 C d_Y(q_1,q_2),
\end{align*}
since $\coth( \rho_z(\beta(t)))<2$ under the conditions on $z$ and $t\in [0,d(q_1,q_2)]$. Hence
\begin{align}\label{eq:d_rhosq}
\left| (d_{q_2}\rho_z(Z(q_2)))^2-(d_{q_1}\rho_z(Z(q_1)))^2 \right| \leq 4C d_Y(q_1,q_2),
\end{align}
since $\abs{d_{q_2}\rho_z(Z(q_2))+d_{q_1}\rho_z(Z(q_1))}\leq 2$, each component being a dual to a unit vector.

Now we split the integral $\abs{\int_{\til{Y}} \left[(d_{q_2}\rho_z(Z(q_2)))^2 -(d_{q_1}\rho_z(Z(q_1)))^2 \right] d\eta_{x_2}^s(z)}$ into the portions on $B$ and $\til{Y}\setminus B$.

 Using \cref{eq:d_rhosq}, we have for the $\til{Y}\setminus B$ portion,
\begin{align*}
 \left| \int_{\til{Y}\setminus B} \left[ (d_{q_2}\rho_z(Z(q_2)))^2 -(d_{q_1}\rho_z(Z(q_1)))^2 \right] d\eta_{x_2}^s(z)\right| \leq 4C d_Y(q_1,q_2).
\end{align*}

For the portion on $B$, using \cref{eq:rho_d_rhosq}, \cref{eq:sigma eta ratio} together with the assumptions $d(q_1,q_2)<1$, we obtain
\begin{align*}
\left| \int_{B} \right. & \left.\rho_z(q_2)(d_{q_2}\rho_z(Z(q_2)))^2 -\rho_z(q_1)(d_{q_1}\rho_z(Z(q_1)))^2 \frac{d\sigma_{x_1}^s(z)}{\norm{\hat{\eta}_{x_1}^{\, s}}} \right|\\
& \leq \int_{B} [1+ 2 C (1+\rho_z(q_1)+d_Y(q_1,q_2))]d_Y(q_1,q_2) \frac{d\sigma_{x_1}^s(z)}{\norm{\hat{\eta}_{x_1}^{\, s}}} \\
&\leq (1+2C)d_Y(q_1,q_2) \frac{\norm{\sigma_{x_1}^s}}{\norm{\hat{\eta}_{x_1}^{\, s}}}+2C d_Y(q_1,q_2)^2 \frac{\norm{\sigma_{x_1}^s}}{\norm{\hat{\eta}_{x_1}^{\, s}}} \\
& \quad + 2C d_Y(q_1,q_2) \int_{\til Y} \rho_z(q_1)\frac{d\sigma_{x_1}^s(z)}{\norm{\hat{\eta}_{x_1}^{\, s}}} \\
&\leq  ((1+2C)(1+C')+2C)d_Y(q_1,q_2)+ 2C(1+C')d_Y(q_1,q_2)^2 \\
&\leq C_0 d_Y(q_1,q_2).
\end{align*}
This completes the estimate.
\end{proof}

Given the existence of the parallel translation in our context, the proof of the lemma below follows similarly to the proof in Lemma A.6 of \cite{Sambusetti-1999} which establishes the corresponding Lemmas 7.5 and 7.5b of \cite{Besson-Courtois-Gallot:95} for the formulation of the barycenter map using the measures $\sigma_x^{s_i}$, as opposed to the calibrating forms used in \cite{Besson-Courtois-Gallot:95}.
However, we need to substitute our version of \cref{lem:H_converge} and its proof instead of Lemma 7.5a in the proof of 7.5b of \cite{Besson-Courtois-Gallot:95}. We present this here together with the other necessary modifications.

\begin{lemma}[cf. Lemma 7.5a of \cite{Besson-Courtois-Gallot:95}]\label{lem:H_implies_deriv}
There is a constant $C\geq 1$ only depending on $N$ such that if $s<h(\bar{X})+1$, $\norm{H_{x}^{s}- \frac1N I}_{op}<\frac13$, and $d_x\til{F_s}$ is defined for $\mc{H}^1$-a.e. $x$ along a geodesic between $x_1$ and $x_2$, then $d(\til{F}_s(x_1),\til{F}_s({x_2}))<Cd(x_1,x_2)$.
\end{lemma}

\begin{proof}
Recall we have
\[
d_x\til{F_s}=s (L_x^s+K_x^s)^{-1}\of A_x^s,
\]
with $K_x^s, L_x^s,$ and $A_x^s$ defined in \eqref{K_x^s}, \eqref{L_x^s}, and \eqref{A_x^s}.

In what follows set $\norm{\cdot}=\norm{\cdot}_{op}$ to be the operator norm with respect to the relevant linear structures on $T_xX$ and $T_{F_{s_i}(x)}Y$ or  $T_{F_{s_j}(x)}Y$ as the context demands, for a regular point $x\in X$.

Hence
\[
\norm{d_x\til{F_{s}}}\leq s \norm{(L_x^{s}+K_x^{s})^{-1}}\norm{A_{x}^{s}}.
\]
Since $K_x^{s}=\left(I-H_{x}^{s}-\sum_{k=1}^{d-1} J_{k} H_{x}^{s} J_{k}\right)$, $L_x$ and $-\sum_{k=1}^{d-1} J_{k} H_{x}^{s} J_{k}$ are positive semidefinite, $\norm{A_{x}^{s_i}}\leq 1$, and $N\geq 3$ we have,
\[
\norm{d_x\til{F_{s}}} \leq 3(h(\bar{X})+1).
\]
Taking $C=3(h(\bar{X})+1)$, the result follows from the fact that $\til{F_{s}}$ is Lipschitz and Lemma~\ref{lem:C-Lip}.
\end{proof}

Set $\Omega$ be the intersection of $\cap_{i\in \N} \Omega^{s_i}$ with the set where the conclusion of \cref{lem:H_converge} holds. Note that $\Omega$ is a full measure subset of $X$.

Given our versions, \cref{lem:H_implies_deriv} and \cref{lem:Parallel_translate}, of Lemmas 7.5a and 7.5b of \cite{Besson-Courtois-Gallot:95}, the proof of the following lemma now follows identically from the proof of Lemma 7.5 in \cite{Besson-Courtois-Gallot:95} except restricted to the set $\Omega$.

\begin{lemma}[cf. Lemma 7.5 of \cite{Besson-Courtois-Gallot:95}]\label{lem:H_uniform}
The endomorphisms $H_{x}^{s_i}$ converge uniformly to $\frac1N I$ on $\Omega\subset X$, as $s_i \rightarrow h(\bar{X})$.
\end{lemma}

\begin{lemma}\label{lem:KL_converge}
For any sequence $s_i\to h(\bar{X})$, the quadratic forms $K^{s_i}_x$ and $L^{s_i}_x$ converge to $\frac{N-2+d}{N} I$ and $0$ respectively uniformly on $\Omega$.
\end{lemma}

\begin{proof}
By \cref{lem:H_uniform}, we have $H^{s_i}_x$ is converging to $\frac{1}{N}I$ uniformly and since $K_x^{s_i}=\left(I-H_{x}^{s_i}-\sum_{k=1}^{d-1} J_{k} H_{x}^{s_i} J_{k}\right)$, the latter approaches $\frac{N-2+d}{N}I$. For $s=s_i$ in \cref{eq:Jac-est1} and \cref{Jac-estimate-1} together with \cref{lem:BCG_est} we have that when $(\Jac F_{s_i})^2$ tends to $(\frac{N}{(N+d-2)^2})^N$ then all of the inequalities tends to equality which implies that $\det(L_x^{s_i}+K_x^{s_i})$ tends to $\det(K_x^{s_i})$. This implies $L_x^{s_i}$ tends to $0$ uniformly since $L_x^{s_i}$ is posititve semi-definite and $K_x^{s_i}$ tends to a multiple of Identity, and thus $L_x^{s_i}$ contributes positively to the denominator unless it is zero (note that $L_x^{s_i}$ and $K_x^{s_i}$ are almost simultaneously diagonalizable).
\end{proof}

\begin{proof}[Proof of Proposition \ref{prop:F_converge}]
As in the proof of \cref{lem:H_implies_deriv}, we have
\[
\norm{d_x\til{F_{s_i}}}\leq s \norm{(L_x^{s_i}+K_x^{s_i})^{-1}}\norm{A_{x}^{s_i}}.
\]
By uniformity of the convergence for any $\eps>0$ we may choose $n$ such that for $i>n$, we have $\norm{(L_x^{s_i}+K_x^{s_i})^{-1}-\frac{N}{h_0}I}\leq \eps$ and $0<s_i-h_0<\eps$ which yields
\[
\norm{d_x\til{F_{s_i}}}\leq (h_0+\eps)(\frac{N}{h_0}+\eps)\norm{A_x^{s_i}}.
\]
Note that we also have $\op{tr} A_x^s=1$ and $\norm{A_x^s}\leq 1$ for a.e. $x\in X$ and all $s>h$ since it is an average of component tensors with this property. Consequently, $\norm{d_x\til{F_{s_i}}}\leq N+1$ for all sufficiently large $i$.

 Hence the $\til{F_{s_i}}$, being Lipschitz by \cref{lem:Lips}, are in fact $(N+1)$-Lipschitz by \cref{lem:C-Lip} for all sufficiently large $i$. Since such a family is pointwise bounded and equicontinuous, there is a convergent subsequence by Arzela-Ascoli. Call this limit map $\til{F}$. (We again denote this convergent subsequence by $\set{s_i}$.)

On the other hand, from \cref{Jac-limit}, we have $\lim_{i\to\infty}\det A_x^{s_i}=\frac{1}{N^N}$ and hence
$\lim_{i\to\infty}\norm{A_x^{s_i}}=\frac{1}{N}$ for a.e. $x\in X$.
 Hence  $\limsup_{i\to\infty}\norm{d_x\til{F_{s_i}}}\leq 1$ for  a.e. $x\in X$.

 In fact the Lipschitz convergence implies convergence in $W^{1,\infty}$ and hence $\norm{d_x\til{F}}\leq 1$ for a.e. $x\in X$. Applying \cref{lem:C-Lip} again shows that $\til{F}$ is $1$-Lipschitz.

Since the family $\set{F_{s}}$ is equicontinuous and converges pointwise, the $F_{s}$ converge uniformly to $F$. By \cref{lem:psitilde}, the $F_{s}$ are homotopic to $f$ and thus so is their uniform limit $F$.

\end{proof}

\providecommand{\bysame}{\leavevmode\hbox to3em{\hrulefill}\thinspace}


\begin{thebibliography}{aa}

\bibitem{AmbrosioColomboDiMarino} L.~Ambrosio, M.~Colombo, S.~Di Marino, \emph{ Sobolev spaces in metric measure spaces: reflexivity and lower semicontinuity of slope,} Variational methods for evolving objects, 1--58, Adv. Stud. Pure Math., 67, Math. Soc. Japan, Tokyo, 2015.

\bibitem{AmbrosioDalMaso90} L.~Ambrosio, G.~Dal Maso, \emph{A general chain rule for distributional derivatives,} Proceedings of the American Mathematical Society, 108(3) (1990), 691--702.


\bibitem{AmbrosioHondaTewodrose} L. Ambrosio, S. Honda, D. Tewodrose, \textit{Short-time behavior of the heat kernel and Weyl's law on $\RCD^*(K, N) $ spaces}, Ann. Glob. Anal. Geom. 53.1 (2018), 97--119.

\bibitem{AGS-08} L. Ambrosio, N. Gigli, G. Savar\'e,  {\it Gradient flows: in metric spaces and in the space of probability measures,} Springer Science \& Business Media, 2008.

\bibitem{Ambrosio-Kirchheim} L. Ambrosio, B.  Kirchheim,  {\it Rectifiable sets in metric and Banach spaces.} Math. Ann. 318, 527--555 (2000).


\bibitem{Berestovskij-Nikolaev:93} V. N. Berestovskij, I. G. Nikolaev, {\it Multidimensional generalized Riemannian spaces,} Geometry, IV, 165--243, 245--250, Encyclopaedia Math. Sci., 70, Springer, Berlin, 1993.

\bibitem{Bess-BCG} L. Bessi\`eres, G.~Besson, G.~Courtois, S.~Gallot, {\it Differentiable rigidity under Ricci curvature lower bound,} Duke Math. J. 161 (2012), no. 1, 29--67.

\bibitem{Besson-Courtois-Gallot:95} G.~Besson, G.~Courtois, S.~Gallot, {\it Entropies et
rigidit\'es des espaces localement sym\'etriques de courbure strictement n\'egative,}
Geom. Funct. Anal. 5 (1995), no. 5, 731--799.

\bibitem{BCG-rend} \bysame , {\it A real Schwarz lemma and some
applications,} Rend. Mat. Appl. (7) 18 (1998), no. 2, 381--410.

 \bibitem{BCG-acta} \bysame , {\it Lemme de Schwarz r\'eel et applications g\'eom\'etriques,} Acta Math. 183 (1999), no. 2, 145--169.

\bibitem{BCG-Samb} G.~Besson, G.~Courtois, S.~Gallot, A.~Sambusetti, {\it Curvature-Free Margulis Lemma for Gromov-Hyperbolic Spaces}, preprint (2017) arxiv:1712.08386.


\bibitem{BCS} J.~Boland, C.~Connell, J.~Souto, {\it Volume rigidity for finite volume manifolds,} Amer. J. Math. 127 (2005), no. 3, 535--550.

\bibitem{brue-semola:18} E. Bru\`e, D. Semola, {\it Constancy of the dimension for $\RCD(K,N)$ spaces via regularity of Lagrangian flows}, Comm. Pure Appl. Math. 73 (2020), no. 6, 1141--1204.
















\bibitem{BrueNaberSemola} E. Bru\`e, A. Naber, D. Semola, {\it Boundary regularity and stability for spaces with Ricci bounded below}, Invent. Math. 228 (2022), no.2, 777--891.

\bibitem{BurBurIva} D.~Burago, Y.~Burago and, S.~V.~Ivanov, \emph{A course in metric geometry}, Graduate Studies in Mathematics, 33. American Mathematical Society, Providence, RI, 2001.




\bibitem{Cheeger} J. Cheeger, {\it Differentiability of Lipschitz functions on metric measure spaces,} Geom. Funct. Anal. 9 (1999), no. 3, 428--517.


\bibitem{Cheeger-Colding} J. Cheeger and T. Colding,  {\em On the structure of spaces with Ricci curvature bounded below {II}}, J. Differential Geom. 46 (1997), no. 3, 406--480.


\bibitem{CobzasMiculescuNicolae} \c{S}. Cobza\c{s}, R. Miculescu, and A. Nicolae, {\it Lipschitz Functions}, Springer International Publishing (2019).

\bibitem{Colding:20} T.~Colding, \emph{Ricci curvature and volume convergence}, Annals of Mathematics 145 (1997), no.~3, 477--501.

\bibitem{Con} C.~Connell, {\it Asymptotic conditions for smooth rigidity of negatively curved manifolds,} Pure Appl. Math. Q. 8 (2012), no. 1, 107--132.

\bibitem{ConSS19} C.~Connell, P.~Su\'arez-Serrato, {\it On higher graph manifolds,} Int. Math. Res. Not. (2019), no. 5, 1281--1311.

\bibitem{stormmod} C.~Connell, X.~Dai, J.~N\'u\~nez-Zimbr\'on, R.~Perales, P.~Su\'arez-Serrato, G.~Wei, {\it Maximal volume entropy rigidity for $\RCD^{*}(-(N-1),N)$ spaces},  Journal London Math. Soc. 104 (2021), 1615--1681.


\bibitem{DJL} M.W.~Davis, T.~Januszkiewicz and J.~Lafont, \emph{4-dimensional locally CAT(0)-manifolds with no Riemannian smoothings,} Duke Mathematical Journal, 161 (2012), 1--28.

\bibitem{DePhilippisGigli-volume cone} G.~De Philippis, N.~Gigli,
{\it
From volume cone to metric cone in the nonsmooth setting}, Geom. Funct. Anal. 26 (2016), no.6, 1526--1587.

\bibitem{DePhilippisGigli} G.~De Philippis, N.~Gigli, {\it Non-collapsed spaces with Ricci curvature bounded from below,} J. \'Ec. polytech. Math. 5 (2018), 613--650.

\bibitem{DE} A. Douady, C. Earle, {\it
Conformally natural extension of homeomorphisms of the circle,}
Acta Math. 157 (1986), no. 1-2, 23--48.



\bibitem{GalazEtAl:18}
F.~Galaz-Garcia, M.~Kell, A.~Mondino, and G.~Sosa, \emph{On quotients of
	spaces with ricci curvature bounded below}, Jour. of Fun. Analysis
275 (2018), no.~6, 1368--1446.


\bibitem{Gigli-diff} N.~Gigli, {\it Nonsmooth differential geometry--An approach tailored for spaces with Ricci curvature bounded from below}, Memoirs of the American Mathematical Society, Vol. 251, Num. 1196, (2018), 161 pp.

\bibitem{GigliMondinoRajala}  N.~Gigli, A.~Mondino and T.~Rajala, {\em Euclidean spaces as weak tangents of infinitesimally Hilbertian metric spaces with Ricci curvature bounded below}, Journal fur die Reine und Angew. Math., Vol.  705, (2015), 233--244.

\bibitem{GigliPasqualetto} N.~Gigli, E.~Pasqualetto, {\it Behaviour of the reference measure on $\mathsf{RCD}$ spaces under charts,} Comm. Anal. Geom. 29 (2021), no. 6, 1391--1414.

\bibitem{GigliPasqualettoBook} N.~Gigli, E.~Pasqualetto, {\it Lectures on Nonsmooth Differential Geometry},  SISSA Springer Series, 2. Springer, Cham, 2020.



\bibitem{Gromov82} M.~Gromov, {\it Volume and Bounded Cohomology}, Pub. Math. I.H.E.S., tome 56, (1982), 5--99.


\bibitem{GroPet} K.~Grove and P.~Petersen, \emph{Manifolds near the boundary of existence}, J. Differential Geom. 33 (1991), no. 2, 379--394.




\bibitem{Hei76} E.~Heintze, \emph{Mannigfaltigkeiten negativer Krummung}, Bonner Mathematische Schriften 350, Bonn, 1976.

\bibitem{Honda-Peng:22} S.~Honda, Y.~Peng, {\it A note on the topological stability theorem from RCD spaces to Riemannian manifolds},
manuscripta math. 172 (2023), 971--1007.

\bibitem{Honda-Sire} S. Honda and Y. Sire, Sobolev mappings between $\RCD$ spaces and applications to harmonic maps: a heat kernel approach, J. Geom. Anal. {\bf 33} (2023), no.~9, Paper No. 272, 87 pp.



\bibitem{Iva} S.~V.~Ivanov,  \emph{Gromov-Hausdorff convergence and volumes of manifolds}, (Russian) ; translated from Algebra i Analiz 9 (1997), no. 5, 65--83 St. Petersburg Math. J. 9 (1998), no. 5, 945--959


\bibitem{Ka} M. Karmanova, {\it Rectifiable sets and coarea formula for metric-valued mappings}, J. Funct. Anal. 254 (2008), 1410--1447.



\bibitem{KapovitchKetterer} V.~Kapovitch, C.~Ketterer, \textit{CD meets CAT},  J. Reine Angew. Math. 766 (2020), 1--44.

\bibitem{KapovitchLytchak21} V.~Kapovitch, A.~Lytchak, {\it Remarks on manifolds with two sided curvature bounds,}  Anal. Geom. Metr. Spaces 9 (2021), no.1, 53--64.

\bibitem{Kapovitch-Mondino}  V. Kapovitch, A. Mondino, {\it On the topology and the boundary of N-dimensional $\RCD (K, N)$ spaces}, Geom. Topol. 25 (2021), no. 1, 445--495.

\bibitem{Kirchheim94} B.~Kirchheim, {\it Rectifiable metric spaces: local structure and regularity of the Hausdorff measure}, Proc. of the American Mathematical Soc. 121.1 (1994), 113--123.

\bibitem{Ketterer} C.~Ketterer, {\it Cones over metric measure spaces and the maximal diameter theorem}, J. Math. Pures Appl. (9) 103 (2015), no. 5, 1228--1275.



\bibitem{Lafont-Schmidt} J.-F.~Lafont, B.~Schmidt, {\it Simplicial volume of closed locally symmetric spaces of non-compact type}, Acta Math.197 (2006), no.1, 129--143.



\bibitem{Li-AlexandrovRigidity} N.~Li, \textit{ Lipschitz-volume rigidity in Alexandrov geometry}, Adv. Math. 275 (2015), 114--146.

\bibitem{LiWang-limits} N.~Li, F.~Wang, {\it Lipschitz-volume rigidity on limit spaces with Ricci curvature bounded from below,} Differential Geometry and its Applications 35 (2014) 50-55.



\bibitem{Lott-Villani09} {J.~Lott, C.~Villani}, {\it Ricci curvature for metric-measure spaces via optimal transport}, Ann. of Math., Vol. 169, (2009), 903--991.

\bibitem{Lloyd} N.~Lloyd. Degree theory, Cambridge Tracts in Math., No. 73,
Cambridge University Press, 1978. vi+172 pp.



\bibitem{M}  A.~Manning, {\it Topological entropy for geodesic flows,} Ann. of Math. (2) 110 (1979), no. 3, 567--573.

\bibitem{Menguy:00} X.~Menguy, {\it Noncollapsing examples with positive Ricci curvature and infinite topological type,} Geom. Funct. Anal. 10 (2000), no. 3, 600--627.

\bibitem{Mer} L.~Merlin, {\it Minimal entropy for uniform lattices in product of hyperbolic planes,} Comment. Math. Helv. 91 (2016), no. 1, 107--129.

\bibitem{Mitsuishi-Yamaguchi:14} A.~Mitsuishi, T.~Yamaguchi, {\it Locally Lipschitz contractibility of Alexandrov spaces and its applications} Pacific J. Math. 270 (2014), no. 2, 393--421.


\bibitem{MN} A.~Mondino, A.~Naber, {\it Structure Theory of Metric-Measure Spaces with Lower Ricci Curvature Bounds I}, Jour. European Math. Soc., Vol.2, No.6, (2019), 1809--185.

\bibitem{Mondino-Wei} A.~Mondino, G.~Wei, \textit{On the universal cover and the fundamental group of an $\RCD^{\ast}(K,N)$-space}, Journal f\"ur die Reine und Angew. Math., No. 753 (2019) 211--237.


\bibitem{Nash} J.~Nash, {\it $C^1$ isometric imbeddings,} Ann. of Math. 60 (1954) 383--396.

\bibitem{Nikolaev:89} I.~G.~ Nikolaev,  \textit{The closure of the set of classical Riemannian spaces,} (Russian) Translated in J. Soviet Math. 55 (1991), no. 6, 2100--2115. Itogi Nauki i Tekhniki, Problems in geometry, Vol. 21 (Russian), 43--46, 216, Akad. Nauk SSSR, Vsesoyuz. Inst. Nauchn. i Tekhn. Inform., Moscow, 1989.

\bibitem{Otsu} Y.~Otsu, \textit{Differential geometric aspects of Alexandrov spaces}, Comparison geometry, MSRI Publications, Vol. 30, (1997), 135--148.


\bibitem{Peters87}  S. Peters, \textit{Convergence of Riemannian manifolds},  Compositio Math., 62(1):3--16, 1987.


\bibitem{Reichel09} L.P.~Reichel, {\it The coarea formula for metric space valued maps}. PhD Thesis, ETH Zurich, 2009.

\bibitem{Reviron} G.~Reviron, {\it Rigidit\'e topologique sous l'hypoth\`ese ``entropie major\'{e}e'' et applications,} Comment. Math. Helv. 83 (2008), no. 4, 815--846.

\bibitem{Ruan22} Y.~Ruan, {\it The Cayley hyperbolic space and volume entropy rigidity}, Math. Z. 306 (2024), no. 1, Paper No. 4, 30 pp.

\bibitem{Sambusetti-1999} A.~Sambusetti, {\it Minimal entropy and simplicial volume}, Manuscripta Math. 99 (1999), 541--560.

\bibitem{Satake} I.~Satake, {\it The Gauss-Bonnet Theorem for V-manifolds,} Journal of the Mathematical Society of Japan, J. Math. Soc. Japan 9(4) (1957), 464--492.

\bibitem{Song} A.~Song, {\it The spherical Plateau problem for group homology}, to appear in S{\'{e}}minaire de th\'{e}orie spectrale et g\'{e}om\'{e}trie, arXiv:2202.10636v1 [math.DG], 2022.

\bibitem{SchroederEtAl:18} V.~Schroeder, H.~Shah, \emph{Almost maximal volume entropy}, Archiv der Mathematik 110 (2018), no.~5, 515--521.

\bibitem{Storm02} P.A.~Storm, {\it  Minimal volume Alexandrov spaces,} J. Differential Geom. 61 (2002), no. 2, 195--225.

\bibitem{St1}  \bysame , {\it The minimal entropy conjecture for nonuniform rank one lattices,} Geom. Funct. Anal. 16 (2006), no. 4, 959--980.

\bibitem{Storm07} \bysame, {\it The barycenter method on singular spaces,} Comment. Math. Helv. 82 (2007), no. 1, 133--173.

\bibitem{Sturm03} { K.T.~Sturm}, {\it Probability measures on metric spaces of nonpositive curvature}, Heat kernels and analysis on manifolds, graphs, and metric spaces (Paris, 2002), 357--390, Contemp. Math., 338, Amer. Math. Soc., Providence, RI, 2003.

\bibitem{Sturm06I} \bysame, {\it On the geometry of metric measure spaces. {I}}, Acta Math., Vol. 196, (2006), 65--131.

\bibitem{Sturm-MM2} \bysame, {\em On the geometry of  metric measure spaces. {II}}, Acta Math., Vol.  196, (2006), 133--177.


\bibitem{Wang2022} J. Wang, {\it $\RCD^*(K,N)$ spaces are semi-locally simply connected},
Journal f\"ur die reine und angewandte Mathematik, no. 806, (2024), pp. 1--7.


\end{thebibliography}
\end{document}